\documentclass[11pt]{amsart}

\usepackage[all]{xy}
\usepackage{hhline}
\usepackage{amssymb}
\usepackage{amsmath,amsthm,amscd,amsfonts}
\usepackage[dvips, dvipsnames, usenames]{color}

\usepackage[active]{srcltx}

\def\k{\Bbbk}

\newcommand{\GL}{\mathbf{GL}}

\newcommand{\ydsn}{{}^{\s_{n}}_{\s_{n}}\mathcal{YD}}

\newcommand{\ot}{\otimes}
\newcommand{\Ind}{\operatorname{Ind}}

\newcommand{\com}{\Delta}

\newcommand\toba{{\mathfrak B }}

\newcommand{\gr}{\operatorname{gr}}
\newcommand{\trid}{\triangleright}

\newcommand{\Lc}{{\mathcal L}}

\newcommand{\eps}{\varepsilon}

\newcommand{\K}{{\mathcal K}}
\newcommand{\Z}{{\mathbb Z}}
\newcommand{\N}{{\mathbb N}}

\def\ot{\otimes}

\def\mD{\mathcal{D}}

\def\mO{\mathcal{O}}

\def\mQ{\mathcal{Q}}
\def\mH{\mathcal{H}}
\def\hq{\mH(\mQ)}

\newcommand{\II}{\mathcal{I}}

\newcommand{\Oc}{{\mathcal O}}
\newcommand{\oc}{{\mathcal O}}
\newcommand{\ydh}{{}^H_H\mathcal{YD}}

\newcommand{\End}{\operatorname{End}}

\newcommand\Hom{\operatorname{Hom}}
\newcommand\Ker{\operatorname{Ker}}

\newcommand{\Ga}{\Gamma}

\newcommand\co{\operatorname{co}}
\newcommand\sg{\operatorname{sg}}
\newcommand{\ydho}{{}^{H_0}_{H_0}\mathcal{YD}}
\newcommand{\ydao}{{}^{A_0}_{A_0}\mathcal{YD}}
\newcommand{\ydga}{{}^{\Ga}_{\Ga}\mathcal{YD}}

\theoremstyle{plain}

\newtheorem{lema}{Lemma}[section]
\newtheorem{theorem}[lema]{Theorem}
\newtheorem{cor}[lema]{Corollary}

\theoremstyle{definition}
\newtheorem{definition}[lema]{Definition}
\newtheorem{exa}[lema]{Example}

\theoremstyle{remark}
\newtheorem{obs}[lema]{Remark}

\theoremstyle{plain}
\newcounter{maint}

\theoremstyle{plain}

\def\ep{\varepsilon}
\newcommand\teta{\tilde{\eta}}

\newcommand\id{\operatorname{id}}

\newcommand\sbb{\mathbb S}

\newcommand\sn{\mathbb S_n}

\newcommand\dm{\mathbb D_m}

\newcommand\s{\mathbb S}
\newcommand{\cS}{\mathcal{S}}
\newcommand{\yddm}{{}^{\dm}_{\dm}\mathcal{YD}}

\def\pf{\begin{proof}}
\def\epf{\end{proof}}

\theoremstyle{remark}

\begin{document}

\renewcommand{\baselinestretch}{1.2}

\thispagestyle{empty}

\title[Deformation by cocycles]
{Deformation by cocycles of pointed Hopf algebras over
non-abelian groups}

\author[Gast\'on A. Garc\'ia, Mitja Mastnak]
{Gast\'on Andr\'es Garc\'ia and Mitja Mastnak}

\thanks{This work was partially supported by
Universidad Nacional de C\'ordoba,
ANPCyT-Foncyt, CONICET,
Ministerio de Ciencia y Tecnolog\'\i a (C\'ordoba) and Secyt (UNC)}

\address{
\newline\noindent Facultad de Matem\'atica, Astronom\'\i a y F\'\i sica,
Universidad Nacional de C\'ordoba.
CIEM -- CONICET.
Medina Allende s/n, Ciudad Universitaria
5000 C\'ordoba,
Argentina
\and
\newline
\noindent Department of Mathematics and C.S.
Saint Mary's University
Halifax, NS B3H 3C3, Canada}
\email{ggarcia@famaf.unc.edu.ar}
\email{mmastnak@cs.smu.ca}


\begin{abstract}
We introduce a method to construct explicitly multiplicative 2-cocycles
for bosoni\-zations of Nichols algebras $\toba(V)$ over Hopf algebras $H$.
These cocycles arise as liftings of $H$-invariant linear functionals
on $V\otimes V$ and give a close formula to deform braided commutator-type 
relations.
Using this construction, we show that all known finite dimensional 
pointed Hopf algebras over the dihedral groups $\dm$ with 
$m=4t\geq 12$, over the symmetric group $\s_{3}$ and some 
families over $\s_{4}$ are
cocycle deformations of bosonizations of Nichols algebras.             
\end{abstract}

\maketitle

\section*{Introduction}
Let $\Bbbk$ be an algebraically closed field of characteristic zero.
A Hopf algebra $A$ is said to be pointed if all simple subcoalgebras are
one dimensional; in particular, its coradical $A_{0}$ coincides 
with the group algebra $\Bbbk G(A)$ over the group of group-like elements.
The best method for classifying finite dimensional pointed Hopf algebras 
over $\Bbbk$ was developed by Andruskiewitsch and Schneider, see \cite{AS}.
Shortly, the method consists on finding first all braided vector spaces $V$
in $\ydao$ such that its Nichols algebra $\toba(V)$ is finite dimensional, then 
find all pointed Hopf algebras $H$ such that 
the graded algebra $\gr H$ induced by the coradical filtration is isomorphic to 
the bosonization $\toba(V)\#A_{0}$, and finally prove the generation in degree
one.

Using this method, they were able to classify in \cite{AS2} 
all finite dimensional pointed Hopf algebras
$A$ such that $G(A)$ is abelian and whose order is relative prime to $210$. 
When the group of group-likes is not abelian, the problem is far from being completed. 
Some hope is present in the lack of examples: in this situation, Nichols algebras tend to be 
infinite dimensional, see for example \cite{AFGV}. 
Nevertheless, examples on which the Nichols algebras are finite dimensional do exist. 
Over the
symmetric groups $\s_{3}$ and $\s_{4}$ these algebras were determined in \cite{AHS}. 
All of them arise from racks associated to a cocycle, and in {\it loc. cit.} and \cite{GG} 
the classification of all finite dimensional pointed Hopf algebras over $\s_{3}$ and $\s_{4}$ 
is completed, respectively. Over the dihedral groups $\dm$ with $m=4t\geq 12$, the 
classification of finite dimensional Nichols algebras and finite dimensional 
pointed Hopf algebras was done in \cite{FG}. In this case, it turns out that  
all Nichols algebras
are isomorphic to exterior algebras. 

Among others, useful tools for constructing new Hopf algebras are the deformation
of the multiplication using multiplicative $2$-cocycles and its dual
notion of deforming the coproduct using twists. With this in mind, it is 
interesting to ask when two new non-isomorphic Hopf algebras are cocycle
deformations of each other. 
It was proved
that the family of finite dimensional 
pointed Hopf algebras over abelian groups $\Ga$ appearing in \cite{AS2} such that
the braided vector space $V$ is a quantum linear space, can be constructed by
deforming the multiplication in $\toba(V)\# \Bbbk \Ga$, see \cite{Mk}, \cite{GM}.
Also, Garc\'ia Iglesias and Mombelli \cite{GIM} proved using module category theory 
the same result for 
all known finite dimensional pointed Hopf algebras over the symmetric 
groups. 

In these notes, we prove in 
Theorem \ref{thm:AI-cocycle} and Theorem \ref{thm:BIL-cocycle} that
all finite dimensional pointed Hopf algebras over the dihedral groups 
$\dm$ with $m=4t\geq 12$ are cocycle deformations of bosonizations
of finite dimensional Nichols algebras, by giving explicitly the
$ 2 $-cocycles. Moreover, using these techniques, we construct the 
cocycles that give the deformation of the bosonizations of Nichols algebras 
over $\s_{3}$ and some families over $\s_{4}$,
see Theorem \ref{thm:q-1-cocycle}.

To construct the $2$-cocycles we apply to non-abelian groups 
the theory developed in \cite{GM}. 
Specifically, let $H$ be a Hopf algebra with bijective antipode and
$V\in \ydh$. 
We first associate to any linear map $\eta:V\ot V\to \Bbbk$ a Hochschild
$2$-cocycle on the Nichols algebras $\toba(V)$. 
If this map is invariant under the action of $H$, it gives rise
to a Hochschild $2$-cocycle $\teta$ on $A=\toba(V)\# H$.
Applying \cite[Lemma 4.1]{GM}, it turns out that the exponentiation $\sigma = e^{\teta}$ 
is indeed a multiplicative $2$-cocycle, provided 
extra conditions are satisfied, \emph{e.g.}, $H$ is 
semisimple and the braiding of $V$ is symmetric.
As 
a consequence, exponentiation of liftings of $\dm$-invariant
linear functionals on Yetter-Drinfeld modules with finite-dimensional Nichols algebra
are multiplicative $2$-cocycles and these are the ones we use to prove
our first main theorems, Theorem \ref{thm:AI-cocycle} and Theorem \ref{thm:BIL-cocycle}.

A strong consequence of using an $H$-invariant linear map 
$\eta$ on $V\ot V$ to 
construct a multiplicative $2$-cocycle $\sigma$ is that 
$\sigma$ coincides with $\eta$ on the elements of $V\ot V$. This 
gives a closed formula for the deformation of \emph{braided commutator}-type
relations $[x,y]_{c}=0$ on $\toba(V)$ for $x,y \in V$ in the same connected
component, which in this case also include the 
\emph{power root vector}-type relations $x^{2}=0$. This simplifies
the computation of the deformation, see Lemma \ref{lem:eta=sigma}. 

The paper is organized as follows: in Section \ref{sec:prelim} we fix the notation, 
make some definitions and recall some facts that are used along the paper, such as
Yetter-Drinfeld modules $V$ over group algebras $H=\Bbbk \Ga$, Nichols algebras
$\toba(V)$, bosonizations $\toba(V)\#H$ between them, 
the lifting method and deforming cocycles. 
In Section \ref{sec:cocycles-boson} we give a method to construct 
Hochschild $2$-cocycles $\teta$ on $\toba(V)\#H$ from 
$H$-invariant linear functionals $\eta$ on $V\ot V$, 
and give necessary conditions on which the exponentiation $\sigma = e^{\teta}$ is
a multiplicative $2$-cocycle on $\toba(V)\#H$.
In Section \ref{sec:pointed-dihedral} we recall the 
classification of finite
dimensional pointed Hopf algebras over the dihedral groups
$\dm$ with $m=4t\geq 12$ and prove that they are cocycle deformation
of bosonizations, using the cocycles constructed in Section \ref{sec:cocycles-boson}. 
Finally in
Section \ref{sec:pointed-symm}, we first recall the notion of racks and
the classification of finite dimensional pointed Hopf algebras over 
$\s_{3}$ and $\s_{4}$, and then prove our last main result,
Theorem \ref{thm:q-1-cocycle} about cocycle deformations
of bosonizations of finite dimensional 
Nichols algebras over the symmetric groups.

\section{Preliminaries}\label{sec:prelim}

\subsection{Conventions}
We work over an algebraically closed field $\k$ of characteristic
zero. Good references for Hopf algebra theory are
\cite{M} and \cite{S}.

If $H$ is a Hopf algebra over $\k$ then $\com$, $\eps$
and $\cS$ denote respectively the comultiplication,
the counit and the antipode. 
Comultiplication and coactions are written using the 
Sweedler-Heynemann notation with summation sign suppressed, \textit{e.g.},
$\com(h)=h_{(1)}\ot
h_{(2)}$ for $h\in H$.

Let $\mathcal{C}$ be a braided monoidal category. A \emph{braided} 
Hopf algebra $R$ in
$\mathcal{C}$ is an object $R$ such that all structural maps
$m_{R}, \com_{R}, \eps_{R}, \cS_{R}$ are morphisms in the category.

The \emph{coradical}
$H_0$ of $H$ is the sum of all simple
sub-coalgebras of $H$. In particular, if $G(H)$ denotes the group
of \emph{group-like elements} of $H$, we have $\k G(H)\subseteq
H_0$. A Hopf algebra is \emph{pointed} if $H_0=\k
G(H)$, that is, all simple sub-coalgebras are one-dimensional.

Denote by $\{H_i\}_{i\geq 0}$ the \emph{coradical
filtration} of $H$; if $H_0$ is a Hopf
subalgebra of $H$, then by $\gr H = \oplus_{n\ge 0}\gr H(n)$ we denote the associated
graded Hopf algebra, with
 $\gr H(n) = H_n/H_{n-1}$, setting
$H_{-1} =0$.

For $h,g\in G(H)$, the linear space of $(h,g)${\it -primitives}
 is:
$$
\mathcal{P}_{h,g}(H) :=\{x\in H\mid\com(x)= x\ot h + g\ot x\}.
$$
If $g=1=h$, the linear space $\mathcal{P}(H) = \mathcal{P}_{1,1}(H)$ is called
the set of primitive elements.

If $M$
is a right $H$-comodule via $\delta(m)=m_{(0)}\ot m_{(1)} \in M\ot H$
for all $m\in M$, 
then the space of {\it right
coinvariants} is $ M^{\co \delta} = \{x\in
M\mid\delta(x)=x\ot1\}$. In particular, if $\pi:H\rightarrow L$ is a morphism of Hopf
algebras, then $H$ is a right $L$-comodule via $(\id\ot\pi)\com$ and
in this case $$H^{\co \pi}:=H^{\co\ (\id\ot\pi)\com}
=  \{h\in
H\mid (\id\ot \pi )\com(h)=h\ot1\}.$$ 
Left coinvariants, written
${^{\co \pi} H }$ are defined analogously.

Let $n,m\in \N$. We denote by $\s_{n}$ the symmetric group 
on $n$ letters and by $\dm$ the dihedral group 
of order $2m$. 

\subsection{Yetter-Drinfeld modules and
Nichols algebras}\label{subsec:ydCG-nichols-alg}
In this subsection we recall the definition of Yetter-Drinfeld modules
over Hopf algebras and we describe the irreducible ones
in case $H$ is a group algebra. We also add the definition of 
Nichols algebras associated to them.

\subsubsection{Yetter-Drinfeld modules over group algebras}
Let $H$ be a Hopf algebra. A \textit{left Yetter-Drinfeld module} $M$ over $H$
is a left $H$-module $(M,\cdot)$ and a left $H$-comodule 
$(M,\delta)$ with $\delta(m)=m_{(-1)}\ot m_{(0)} \in H\ot M$ 
for all $m\in M$, satisfying the compatibility 
condition 
$$ \delta(h\cdot m) = h_{(1)}m_{(-1)}\cS(h_{(3)})\ot h_{(2)}\cdot m_{(0)}
\qquad \forall\ m\in M, h\in H. $$
We denote by $\ydh$ the category of left Yetter-Drinfeld modules over $H$. 
It is a braided monoidal category: for any $M, N \in \ydh$, 
the braiding $c_{M,N}:M\ot N \to N\ot M$ is given by 
$$c_{M,N}(m\ot n) = m_{(-1)}\cdot n \ot m_{(0)} \qquad \forall\  
m\in M,n\in N.$$
Assume $H=\k \Ga$ with 
$\Ga$ a finite group. A left Yetter-Drinfeld module over $\k \Ga$ is a left
$\k\Ga$-module and left $\k \Ga$-comodule $M$ such that
$$
\delta(g.m) = ghg^{-1} \otimes g.m, \qquad \forall\ m\in M_h, g, h\in \Ga,
$$
where $M_h = \{m\in M: \delta(m) = h\otimes m\}$; clearly, $M =
\oplus_{h\in \Ga} M_h$ and thus $M$ is a $\Ga$-graded module.  
We denote this category simply 
by $\ydga$.

Yetter-Drinfeld modules over $\k\Ga$ are
completely reducible. The irreducible modules
are parameterized by pairs $(\oc, \rho)$, where $\oc$ is a
conjugacy class of $\Ga$ and $(\rho,V)$ is an irreducible representation of the
centralizer $C_\Ga(\sigma)$ of a fixed point $\sigma\in \oc$. 
The corresponding Yetter-Drinfeld module is given by
$M(\oc, \rho) = \Ind_{C_\Ga(\sigma)}^{G}V =
\Bbbk \Ga \otimes_{C_\Ga(\sigma)} V$. 

Explicitly, let $\sigma_1 = \sigma$, \dots, $\sigma_{n}$
be a numeration of $\oc$ and
let $g_i\in \Ga$ such that $g_i \sigma g_{i}^{-1} = \sigma_i$ for all $1\le i \le
n$. Then  $M(\oc, \rho) = \oplus_{1\le i \le n}g_i\otimes V $. The Yetter-Drinfeld module 
structure is given as follows: let
$g_iv := g_i\otimes v \in M(\oc,\rho)$, $1\le i \le n$, $v\in V$.
If $v\in V$ and $1\le i \le n$, then the action of $g\in \Ga$ and
the coaction are given by
$$g\cdot (g_iv) = g_j(\gamma\cdot v), \qquad
\delta(g_iv) = \sigma_i\otimes g_iv,
$$
where $gg_i = g_j\gamma$, for some $1\le j \le n$ and $\gamma\in
C_\Ga(\sigma)$. 
In this case,
the explicit formula for the braiding is given by
$$
c(g_iv\otimes g_jw) = \sigma_i\cdot(g_jw)\otimes g_iv = g_h(\gamma\cdot
v)\otimes g_iv,$$ 
	for any $1\le i,j\le n$, $v,w\in V$,
where $\sigma_ig_j = g_h\gamma$ for unique $h$, $1\le h \le n$ and
$\gamma \in C_\Ga(\sigma)$. For explicit examples with $\Ga = \dm$
and $\Ga = \sn$
see Subsections \ref{subsec:ydnadm} and \ref{subsec:ydnasn}.

\subsubsection{Nichols algebras}\label{subsubsec:nichols-alg}
The Nichols algebra of a braided vector space $(V,c)$ can be
defined in various different ways, see \cite{N, AG, AS, H}. They are connected
Hopf algebras in a braided monoidal category with certain properties.
In general, the computation of the Nichols algebra of an arbitrary braided
vector space is a delicate issue. We are interested in Nichols algebras of
braided vector spaces arising from Yetter-Drinfeld modules, hence
we give the explicit definition in this case.

The notion of Nichols algebras appeared first in the work 
of Nichols \cite{N} and then it was later rediscovered by several authors. 
\begin{definition}\cite[Def. 2.1]{AS}
Let $H$ be a Hopf algebra and $V \in \ydh$. A braided $\N$-graded 
Hopf algebra $R = \bigoplus_{n\geq 0} R(n) \in \ydh$  is called 
the \textit{Nichols algebra} of $V$ if 
\begin{enumerate}
 \item[(i)] $\k\simeq R(0)$, $V\simeq R(1) \in \ydh$,
 \item[(ii)] $R(1) = \mathcal{P}(R)
=\{r\in R~|~\com_{R}(r)=r\ot 1 + 1\ot r\}$, the linear space of primitive elements.
 \item[(iii)] $R$ is generated as an algebra by $R(1)$.
\end{enumerate}
In this case, $R$ is denoted by $\toba(V) = \bigoplus_{n\geq 0} \toba^{n}(V) $.    
\end{definition}

For any $V \in \ydh$ there is a Nichols algebra $\toba(V)$ associated to it. 
It can be constructed as a quotient of the tensor algebra $T (V)$ by the largest 
homogeneous two-sided ideal $I$ satisfying:
\begin{itemize}
\item $I$ is generated by 
homogeneous elements of degree $\geq 2$.
\item $\com(I) \subseteq I\ot T(V) + T(V)\ot I$, i.~e., it is also a coideal.
\end{itemize}
In such a case, $\toba(V) = T(V)/ I$. See
\cite[Section 2.1]{AS} for details.

An important observation is that the Nichols algebra 
$\toba(V)$, as algebra and coalgebra,
is completely determined just by the braiding. 

Given a 
braided vector space $(V,c)$, one may construct the 
Nichols algebra $\toba(V,c) = \toba(V)$ in a similar way as before
by taking a quotient of the tensor algebra $T(V)$ by the homogeneous 
two-sided ideal given by the kernel of an homogeneous symmetrizer. 
Briefly, let $\mathbb{B}_{n}$ be the braid
group of $n$ letters.
Since $c$ satisfies the braid
equation, it induces a representation of 
$\mathbb{B}_{n}$, $\rho_{n}:\mathbb{B}_{n} \to \GL(V^{\ot n})$
for each $n \geq 2$. 
Consider the morphisms
$$Q_{n} = \sum_{\sigma \in \s_{n}}\rho_{n}(M (\sigma)) \in \End(V^{\ot n} ),$$
where $M:\s_{n} \to \mathbb{B}_{n}$ is the Matsumoto section
corresponding to the canonical projection 
$\mathbb{B}_{n}\twoheadrightarrow \s_{n}$.
Then the Nichols algebra $\toba(V)$ 
is the quotient of the tensor algebra $T (V)$ by the two-sided ideal
$\mathcal{J}=\bigoplus_{n\geq 2 } \Ker Q_{n}$. 
If $c = \tau$ is the usual flip, then
$B(V ) = S V$ is just the symmetric algebra of $V$; 
if $c = -\tau$, then $\toba(V ) = \bigwedge V$ is the exterior
algebra of $V$. 

Let $\Ga$ be a finite group. 
We denote by
$\toba(\oc,\rho)$ the 
Nichols algebra associated
to the irreducible Yetter-Drinfeld module $M(\oc, \rho) \in \ydga$.

Let $V \in \ydga$ such that 
$\toba(V)$ is finite dimensional and
let $\{x_{i}\}_{i\in I}$ be homogeneous primitive elements 
that span linearly $V$ with $\delta (x_{i}) = g_{i}\ot x_{i}$
and $g_{i}\in \Ga$ for all $i\in I$.
Since for all $h\in \Ga$,
$h\cdot x_{i}$ is again an homogeneous primitive element, from now on
we assume that
\begin{equation}\label{eq:assumption-yd}
h\cdot x_{i} = \chi_{i}(h) x_{\sigma(h)(i)}\text{ for all }i\in I,\ h\in \Ga, 
\end{equation}
where $\sigma: \Ga \to \sbb_{I}$ and $\chi_{i} : \Ga \to \Bbbk$ 
is a character, see \cite[Ex. 5.9]{AG2}. This condition holds for
all finite dimensional pointed Hopf algebras over the symmetric
groups and over the dihedral groups, see Sections \ref{sec:pointed-dihedral} 
and \ref{sec:pointed-symm}.
We write $\sigma(g_{i})(j) = i\trid j$ for all $i,j \in I$.

\begin{obs}\label{rmk:rel-racks}
$(a)$  If $V$ is irreducible, then $V \simeq M(\Oc, \rho)$ with
$\Oc$ a conjugacy class of $\Ga$. In such a case,
$I$ can be identified with $\Oc$ and $\sigma (h)$
is just the conjugation by $h$.

$(b)$ If $\Ga$ is abelian, then $V$ is a 
braided space of diagonal type, \textit{i.\ e.\ }
$h\cdot x_{i} = \chi_{i}(h) x_{i}$ for all
$i\in I$ and $\sigma (h) = \id$ for all $h\in \Ga$.
\end{obs}

Important examples of Nichols algebras come from the 
theory of quantum groups. As shown in \cite{AS2}, they play
a crucial role in the classification of finite dimensional 
pointed Hopf algebras over $\Bbbk$, via the lifting method.
For more details 
on Nichols algebras see \cite{AS}, \cite{H} and references therein; and for
explicit examples with $\Ga = \dm$
and $\Ga = \sn$ see  Subsections \ref{subsec:ydnadm} and \ref{subsec:ydnasn}.

\subsection{Bosonization and Hopf algebras with a projection}
Let $H$ be a Hopf algebra and $R$ a braided Hopf algebra in $\ydh$. 
The procedure to obtain a 
usual Hopf algebra from $R$ and $H$ is called 
the Majid-Radford product or \emph{bosonization}, and it is usually
denoted by $R \#H$. As vector spaces $R \# H = R\otimes H$, and the
multiplication and comultiplication are given by the smash-product 
and smash-coproduct,
respectively. That is, for all $r, s \in R$ and $g,h \in H$
\begin{align*}
(r \# g)(s \#h) & = r(g_{(1)}\cdot s)\# g_{(2)}h,\\ 
\com(r \# g) & =r^{(1)} \# (r^{(2)})_{(-1)}g_{(1)} \ot 
(r^{(2)})_{(0)}\# g_{(2)},
\end{align*}
where $\com_{R}(r) = r^{(1)}\ot r^{(2)}$ denotes the comultiplication in $R\in \ydh$.
If $r\in R$ and $h\in H$, then we identify $r=r\# 1$ and  
$h=1\# h$; in particular we have $rh=r\# h$ and $hr=h_{(1)}\cdot r\# h_{(2)}$.
Clearly, the map $\iota: H \to R\#H$ given by $\iota(h) = 1\#h$ for all
$h\in H$ is an injective Hopf algebra map, and the map $\pi: R\#H \to H$ 
given by $\pi(r\#h) = \eps_{R}(r)h$ for all $r\in R$, $h\in H$
is a surjective Hopf algebra map such that $\pi \circ \iota = \id_{H} $. 
Moreover, it holds that $R = (R\#H)^{\co \pi}$. 

Conversely, let $A$ be a Hopf algebra with bijective antipode and
$\pi: A\to H$ a Hopf algebra epimorphism admitting 
a Hopf algebra section $\iota: H\to A$ such that $\pi\circ\iota =\id_{H}$.
Then $R=A^{\co\pi}$ is a braided Hopf algebra in $\ydh$ and $A\simeq R\# H$
as Hopf algebras.

\subsection{The Lifting Method}\label{subsec:lifting-method}
The \textit{Lifting Method} was introduced by Andruskiewitsch and
Schneider, see \cite{AS}. It is one of the
more general results concerning the structure of Hopf algebras and 
became a powerful 
tool for classifying finite dimensional pointed Hopf algebras, as shown 
in \cite{AS2}, \cite{AHS}, \cite{GG},  \cite{FG}, among others.
 
Let $H$ be a finite dimensional pointed Hopf algebra with
coradical $H_{0} = \k G(H)$ and
$\gr H = \oplus_{n\ge 0}\gr H(n)$ the associated
graded Hopf algebra.
If $\pi:\gr H\to H_0$ denotes the homogeneous
projection, then $R= (\gr H)^{\co \pi}$ is called the \emph{diagram} of
$H$. It is a braided Hopf algebra in $\ydho$ and it is a
graded sub-object of $\gr H$. The linear space
$R(1)$, with
the braiding from $\ydho$, is called the \emph{infinitesimal
braiding} of $H$ and coincides with the subspace of primitive
elements $\mathcal{P}(R)$.
This braiding is the key to the structure of pointed 
Hopf algebras.
It turns out that the Hopf algebra $\gr H$ is the
bosonization
$\gr H \simeq R\# \k G(H)$ and the
subalgebra of $R$ generated by $V$ is isomorphic to the Nichols
algebra $\toba(V)$, see Subsection \ref{subsec:ydCG-nichols-alg}.

Let $\Ga$ be a finite group. 
The main steps of the lifting method are:
\begin{itemize}
\item determine all $V\in \ydga$ such that the 
Nichols algebra $\toba(V)$ is finite dimensional,

\item for such $V$, compute all Hopf algebras $H$ such
that $\gr H\simeq \toba(V)\# \k\Ga$. We call $H$ a
\emph{lifting} of $\toba(V)$ over $\Ga\simeq G(H)$.

\item Prove that any finite dimensional pointed Hopf algebras with group 
of group-likes $\Ga$
is generated by group-likes and skew-primitives.
\end{itemize}

Using this method, it was possible to classify finite dimensional
nontrivial pointed Hopf algebras over abelian groups with order 
relative prime to $210$, over the symmetric groups $\s_{3}$ and 
$\s_{4}$ and over the dihedral groups $\dm$ with $m=4t$, $t\geq 3$, see
\textit{loc.\ cit.} 

\subsection{Deforming cocycles}\label{subsec:def-cocycles}
Let $A$ be a Hopf algebra.
Recall that a convolution invertible linear map 
$\sigma $ in $\Hom_{\Bbbk}(A\ot A, \Bbbk)$
is a  
\textit{normalized multiplicative 2-cocycle} if 
$$ \sigma(b_{(1)},c_{(1)})\sigma(a,b_{(2)}c_{(2)}) =
\sigma(a_{(1)},b_{(1)})\sigma(a_{(2)}b_{(2)},c)  $$
and $\sigma (a,1) = \eps(a) = \sigma(1,a)$ for all $a,b,c \in A$,
see \cite[Sec. 7.1]{M}.
In particular, the inverse of $\sigma $ is given by 
$\sigma^{-1}(a,b) = \sigma(\cS(a),b)$ for all $a,b\in A$ and the 2-cocycle
condition is equivalent to 
\begin{equation}\label{eq:2-cocycle}
(\eps  \ot \sigma) * \sigma ( \id_{A} \ot m) = (\sigma \ot \eps) * \sigma (m \ot \id_{A}) 
\end{equation}
and $\sigma (\id_{A}, 1) = \eps = \sigma(1,\id_{A})$. 

Using a $2$-cocycle $\sigma$ it is possible to define 
a new algebra structure on $A$ by deforming the multiplication,
which we denote by $ A_{\sigma} $. Moreover, $A_{\sigma}$ is indeed
a Hopf algebra with 
$A = A_{\sigma}$ as coalgebras,  
deformed multiplication
$m_{\sigma} = \sigma * m * \sigma^{-1} : A \ot A \to A$
given by $$m_{\sigma}(a,b) = a\cdot_{\sigma}b = \sigma(a_{(1)},b_{(1)})a_{(2)}b_{(2)}
\sigma^{-1}(a_{(3)},b_{(3)})\qquad\text{ for all }a,b\in A,$$ 
and antipode
$\cS_{\sigma} = \sigma * \cS * \sigma^{-1} : A \to A$ given by (see \cite{doi} for details)
$$\cS_{\sigma}(a)=\sigma(a_{(1)},\cS(a_{(2)}))\cS(a_{(3)})
\sigma^{-1}(\cS(a_{(4)}),a_{(5)})\qquad\text{ for all }a\in A.$$

\subsubsection{Deforming cocycles for graded Hopf algebras}
Assume now $A=\oplus_{n\geq 0} A_n$ is a graded Hopf algebra and let 
$\eta \in \Hom_{\Bbbk}(A\otimes A,\Bbbk)$ be a Hochschild
2-cocycle on $A$, that is 
$$\eps(a)\eta(b,c)+\eta(a,bc) = \eta(a,b)\eps(c)+\eta(ab,c) \qquad \text{for all }
a,b,c\in A.$$
If we assume further that $\eta|_{A\ot A_0+A_0\ot A}=0$,
then
$$
\sigma=e^\eta=\sum_{i=0}^\infty \frac{\eta^{* i}}{i!}\colon A\otimes A\to \Bbbk
$$
is a well-defined convolution invertible map with convolution
inverse $e^{-\eta}$; moreover, 
\textit{often}
$e^\eta$ will be a multiplicative 2-cocycle. For
instance, this happens whenever $\eta(\id\ot m)$ and $\eta(m\ot \id)$ commute 
(with respect to the convolution product) with $\eps\ot \eta$ and
$\eta\ot \eps$, respectively. Also note that if $\eta*\eta=0$, then
$e^\eta=\eps+\eta$.

The following result is well-known in the cocommutative setting \cite{Sw} and 
it is a generalization of a result from \cite{GZ}. Although the result
holds for more general graded algebras, we state it for the case of our interest where
$A=\toba(V)\#H$ is a graded Hopf algebra given by the 
bosonization of a Nichols algebra $\toba(V)$
with a Hopf algebra $H$.

\begin{lema}\cite[Lemma 4.1]{GM}\label{le:cond-comm-cocycle}
Let $A=\toba(V)\#H$ be a bosonization of a Nichols algebra $\toba(V)$
with a Hopf algebra $H$.
If $\eta\in \Hom_{\Bbbk}(A\otimes A,\Bbbk)$ is a Hochschild $2$-cocycle such that 
$\eta(\id\ot m)$ commutes with $\eps\ot \eta$ and
$\eta(m\ot \id)$ commutes with $\eta\ot\eps$ in the convolution algebra 
$\Hom_\Bbbk(A\otimes A\otimes A, \Bbbk)$, then
$e^\eta$ is a multiplicative $2$-cocycle.
\qed
\end{lema}

\section{Cocycles on bosonizations of Nichols algebras}\label{sec:cocycles-boson}
In this section we discuss sufficient conditions for a Hochschild $2$-cocycle on
$\toba(V)\#H$ to
satisfy the conditions of Lemma \ref{le:cond-comm-cocycle}; and thus inducing
a multiplicative cocycle via the exponential map. 
We define it by 
extending a linear functional on $V\ot V$ invariant under the action of $H$.

From now on we assume that $A=\toba(V)\# H$, where
$H$ is Hopf algebra with bijective antipode, $V \in \ydh$ and 
$\toba(V)\in\ydh$ is the Nichols algebra of $V$.  
Since $\toba(V)$ is graded, we have that $A$ is also graded
with the gradation given by $A_{0}= H$ and $A_{n}= \toba^{n}(V)\# H$.

\subsection{Hochschild cocycles on $\toba(V)\#H$ from 
$H$-invariant lineal functionals on $V\ot V$}\label{subsec:hoch-H-inv}
Let $\eta:V\ot V \to \Bbbk $ be a linear map. Then we can define a Hochschild
$2$-cocycle on $\toba(V)$ by 
$$\eta(\toba^{m}(V)\otimes \toba^{n}(V))=0
\text{ if }(m,n)\not= (1,1).$$ 
Such a map is called an \emph{$\eps$-biderivation}, since 
it is an $\eps$-derivation on each variable, that is, 
we have $\eta(1,-)=0=\eta(-,1)$ and
$\eta(xy,-)=0=\eta(-,xy)$ for all $x,y\in \toba(V)$ 
such that $\ep(x)=0=\ep(y)$.

Since $H$ acts on $V$, we have that 
$ H $
acts on the set of all linear maps
$ \eta: V\ot V \to \Bbbk $ by 
$ \eta^{h}(x,y) =  \eta(h_{(1)}\cdot x, h_{(2)}\cdot y)$ for all
$ h\in H, x,y\in V $. We say that
$ \eta $ is \emph{$ H $-invariant} if
$ \eta^{h}= \eta $ for all $ h\in H $.

The following lemma tell us how to construct a Hochschild $2$-cocycle
on $A$ from an $H$-invariant linear functional on $V\ot V$. Its proof is straighforward 
and it is left to the reader.
\begin{lema}\label{lem:funtional-2cocycle}
Let $\eta:V\ot V \to \Bbbk$ be an $H$-invariant linear map. Then the 
map $\teta\colon A\ot A\to\k$ defined by  
$\teta(A_m\otimes A_n)=0$ if $(m,n)\not= (1,1)$ 
and 
$$\teta(x\# h,y\# h')=\eta(x, h\cdot y) \varepsilon(h') \text{ for 
all }x,y\in V \text{ and }h,h'\in H,$$
is an $H$-invariant Hochschild $2$-cocycle on $A$ that satisfies 
$\tilde{\eta}|_{A_{0}\ot A + A\ot A_{0}} =0 $.\qed
\end{lema}

\begin{obs}
Assume $V$ is finite dimensional and let   
$(x_i)_{i\in I}$ be a basis of $V$ and $(d_i)_{i\in I}$ the dual basis of
$V^*$. Then 
any linear combination of the tensor products
$d_i\otimes d_j$ induces 
a Hochschild 2-cocycle on $\toba(V)$. 
In particular, by \eqref{eq:assumption-yd} the map
$\eta = \sum_{i,j\in I} a_{ij}d_{i}\ot d_{j}$ 
is $\Ga$-invariant if 
$$
a_{k,\ell} = 
\chi_{k}(g)\chi_{\ell}(g)a_{\sigma(g)(k),\sigma(g)(\ell)} 
\text{ for all } g\in \Ga, k,\ell \in I.
$$
\end{obs}

\subsection{On the commuting conditions of $H$-invariant Hochschild $2$-cocycles}
In the remaining of this section we study when an exponentiation 
$e^{\tilde{\eta}}$ of a Hochschild $2$-cocycle induced by an $H$-invariant
lineal functional $\eta$ on $V\ot V$ is a multiplicative $2$-cocycle. In particular, 
it is always the case if $H$ is semisimple and the braiding of $V$ is symmetric.

Denote by $ c $ the braiding of $ V \in \ydh$. It induces an action of the braid group $\mathbb{B}_{n}$
on $V^{\ot n}$. If $\pi \in \mathbb{B}_{n}$, we denote by $c_\pi:
V_1\otimes\ldots\otimes V_n \to V_{\pi(1)}\otimes\ldots\otimes V_{\pi(n)}$ the map
induced by this action.  In particular we write
\begin{align*}
&c_{231} = (\id\ot c)(c\ot \id),\quad
c_{1324} = (\id\ot c\ot \id),\quad
c_{2413} = (\id\ot c\ot \id)(c\ot c),\\
&c_{1423} = (\id\ot c\ot \id)(\id\ot \id\ot c)\mbox{ and }
c_{2314} = (\id\ot c\ot \id)(c\ot\id\ot \id).
\end{align*}

Let $\eta:V\ot V\to \Bbbk$ be an $H$-invariant linear functional and denote
again by $\eta$ the map defined on $\toba(V)$, then we have that
\begin{equation}
c(\id\ot\eta)=(\eta\ot\id)(1\ot c)(c\ot 1) = (\eta\ot\id)c_{231} , \label{eq5}
\end{equation}
that is, for all $a,b,c\in V$ we have:
\begin{equation}\label{eq6}
\eta(b,c)\otimes a = \eta(a_{(-1)}b,c)\otimes a_{(0)}   
 = (\eta\ot\id)c_{231}(a\ot b\ot c).
\end{equation}
Indeed, $(\eta\ot\id)c_{231}(a\ot b\ot c) 
= \eta(a_{(-2)}\cdot b, a_{(-1)}\cdot c)\otimes a_{(0)} = 
\eta^{a_{(-1)}}(b,c)\otimes a_{(0)} =
\eta(b,c)\otimes \ep(a_{(-1)})a_{(0)} = 
c(\id\ot\eta)(a\ot b\ot c) $ for all $a,b,c\in V$.

Using the definition of $\teta$ we also have 
for all $a,b\in A$ and $c\in V$ that
\begin{equation}\label{eq7}
\teta(a,b)\ot c = \teta(a,bc_{(-1)})\ot c_{(0)}.
\end{equation}

The following lemma shows that both conditions on 
the commutativity of the maps in Lemma 
\ref{le:cond-comm-cocycle}, $(b)$ and $(c)$ below, are equivalent. Note
that, as a consequence, we need only to verify equalities on $V^{\ot 4}$.

\begin{lema} \label{lem:cond-eps-biderivations} 
Let $\eta:V\ot V\to \Bbbk$ be an $H$-invariant linear functional.
The following are equivalent:
\begin{enumerate}
\item[$(a)$] The following conditions hold on $V^{\ot 4}$:
\begin{eqnarray}
(\eta\otimes \eta) c_{1324} &=& (\eta\otimes\eta) 
c_{2413}\mbox{ and }  \label{eq1}\\
(\eta\otimes \eta) c_{1423} &=& (\eta\ot\eta) c_{2314}. \label{eq2}
\end{eqnarray}
\item[$(b)$] The following condition holds on $A^{\ot 3}$:
$(\ep\ot\teta)*\teta(\id\ot m) =
 \teta(\id\ot m) * (\ep\ot\teta)$.
\item[$(c)$] The following condition holds on $A^{\ot 3}$:
$(\teta\ot\ep)*\teta(m\ot\id) = 
\teta(m\ot\id)*(\teta\ot\ep)$.
\end{enumerate}
\end{lema}

\begin{proof} 
Throughout the proof we use that fact that for $v\in V\subseteq A$ we have
$\Delta(v) = v_{(1)}\otimes v_{(2)} = v_{(-1)}\ot v_{(0)}+v\ot 1$.
We first show that $(a)$ is equivalent to $(b)$.  
Note that it is both necessary and sufficient to verify 
$(b)$ by evaluating at $\toba(V)_m\ot\toba(V)_n\ot\toba(V)_p$ 
for all $(m,n,p)$.  Unless $(m,n,p)=(1,2,1)$ or $(m,n,p)=(1,1,2)$ 
both sides trivially give $0$.
Now evaluation at $a\otimes b\otimes cd$ for $a,b,c,d\in V$ yields
that $\teta(b, (cd)_{(1)})\teta(a, (cd)_{(2)}) = 
\teta(b_{(0)}, (cd)_{(2)})\teta(a, b_{(-1)}(cd)_{(1)})$.
After expanding $\Delta(cd)$ and using the fact that 
$\eta$ is $H$-invariant we get
$$
\eta(b,c)\eta(a,d)+\eta(b, c_{(-1)}\cdot d)\eta(a,c_{(0)})  =
\eta(b_{(0)},d)\eta(a, b_{(-1)}\cdot c) + \eta(b_{(0)},c_{(0)})
\eta(a, (b_{(-1)}c_{(-1)})\cdot d).
$$
Since by $(b)$ we have that 
$(\eta\otimes\eta)(a,d,b,c)=(\eta\ot\eta)c_{2314}(a,b,c,d)$,
this is equivalent to 
\begin{equation}\label{eqone}
(\eta\otimes \eta)(c_{2314}+c_{2413}-c_{1324}-c_{1423})=0.
\end{equation}
We now examine evaluation of $(b)$ at 
$a\otimes bc\otimes d$ for $a,b,c,d\in V$.  
Using (\ref{eq7}) we get:
\begin{eqnarray}\label{eqthree}
\teta((bc)_{(1)},d)\teta(a, (bc)_{(2)}) = \teta((bc)_{(2)}, d)\teta(a,(bc)_{(1)}).
\end{eqnarray}
Expanding $\Delta(bc)$ and using equations \eqref{eq6} 
and \eqref{eq7} as well as the definition of $\teta$ we get
\begin{equation*}
\eta(b, c_{(-1)}\cdot d)\eta(a, c_{(0)})+\eta(c,d)\eta(a,b) = 
\eta(c,d)\eta(a,b) + \eta(b_{(0)},d)\eta(a, b_{(-1)}\cdot c),
\end{equation*}
which simplifies to
$\eta(b, c_{(-1)}\cdot d)\eta(a, c_{(0)}) = \eta(a, b_{(-1)}\cdot c)\eta(b_{(0)},d)$,
or equivalently, 
$
(\eta\ot\eta)c_{2413} = (\eta\ot\eta)c_{1324}$, which is 
exactly \eqref{eq1}.
Hence we have that $(b)$ is equivalent to \eqref{eqone} and \eqref{eq1}, and
consequently to $(a)$, 
since   
\eqref{eqone} and \eqref{eq1} 
give \eqref{eq2}.

The equivalence of $(a)$ and $(c)$ works almost exactly the same way.  
For the sake of completeness we provide some intermediate steps.  
We verify $(c)$ by evaluating it at $a\otimes bc\otimes d$ 
and $ab\otimes c\otimes d$ for $a,b,c,d\in V$.  Evaluation at
$a\otimes bc\otimes d$ yields \eqref{eqthree} which is equivalent to \eqref{eq1}.  
Evaluation at $ab\otimes c\otimes d$ and using that 
$\Delta(c) = c\ot 1 + c_{(-1)}\ot c_{(0)}$ yields 
\begin{equation*}
\teta((ab)_{(1)},c)\teta((ab)_{(2)},d) = \teta((ab)_{(2)},c_{(0)})
\teta((ab)_{(1)}, c_{(-1)}\cdot d).
\end{equation*}
This simplifies to
$$
\eta(a, b_{(-1)}\cdot c)\eta(b_{(0)},d)+\eta(b,c)\eta(a,d) = 
\eta(b_{(0)}, c_{(0)})\eta(a, b_{(-1)}c_{(-1)}\cdot d)+
\eta(a,c_{(0)})\eta(b,c_{(-1)}\cdot d),
$$
which by \eqref{eq6}  can be written as
$
(\eta\ot\eta)(c_{1324}+c_{2314}-c_{1423}-c_{2413})=0
$.
We conclude the proof by noting that this equation together with 
\eqref{eq1} are equivalent to (a). 
\end{proof}

For the following lemma, observe that using $ c_{1324} = \id\ot c\ot \id$
and $ c_{2413} = (\id\ot c \ot \id) (c\ot c)$, we get that  \eqref{eq1}
is equivalent to 
$(\eta\ot\eta)(\id\ot c\ot \id) = (\eta\ot\eta)(\id\ot c\ot \id)(c\ot c)$.

The next two results state that the 
conditions in Lemma \ref{lem:cond-eps-biderivations}$(a)$
are always fulfilled when $H$ is semisimple and the braiding is symmetric.

\begin{lema} \label{lemma-involution-1}
If $S_{H}^{2}=\id_{H}$, 
then \eqref{eq1} is equivalent to either of the following equations on $V^{\ot 4}$:
\begin{eqnarray}
(\eta\ot\eta)(\id\ot c\ot \id) &=& (\eta\ot\eta)(\id\ot c^{-1} \ot \id), \\
(\eta\ot\eta)(\id\ot c^2\ot \id) &=& \eta\ot\eta.
\end{eqnarray}
In particular, \eqref{eq1} is always satisfied when $c^{2}=\id_{V\ot V}$ 
and $H$ is semisimple.  
\end{lema}

\begin{proof} Note 
that in the case $S$ is an involution, and therefore for $h\in H$ we have
$\ep(h)=S(h_{(2)})h_{(1)}$ (this is used for going from line 
8 to line 9 and for going to the last line from the line above it), we get for all
$a,b,c,d \in $:
\begin{align*}
& (\eta\ot\eta)c_{2413}(a, b, c, d) = (\eta\ot \eta)(\id\ot c\ot\id)
(a_{(-1)}\cdot b, a_{(0)}, c_{(-1)}\cdot d, c_{(0)})\\
&=  \eta(a_{(-2)}\cdot b, a_{(-1)}c_{(-1)}\cdot d)\eta(a_{(0)}, c_{(0)})
= \eta^{a_{(-1)}}(b, c_{(-1)}\cdot d)\eta(a_{(0)}, c_{(0)})
= \eta(b, c_{(-1)}\cdot d)\eta(a, c_{(0)})\\
&
= \eta^{S(c_{(-1)})}(b, c_{(-2)}\cdot d)\eta(a, c_{(0)})
= \eta(S(c_{(-1)})_{(1)}\cdot b, (S(c_{(-1)})_{(2)}c_{(-2)})\cdot d)\eta(a, c_{(0)})\\
& =  \eta(S(c_{(-1)_{(2)}})\cdot b, (S(c_{(-1)_{(1)}})c_{(-2)})
\cdot d)\eta(a, c_{(0)})
=  \eta(S(c_{(-1)})\cdot b, (S(c_{(-2)})c_{(-3)})\cdot d)\eta(a, c_{(0)})\\
& =\eta(S(c_{(-1)})\cdot b, d)\eta(a, c_{(0)})=
 \eta(a, c_{(0)})\eta(S(c_{(-1)})\cdot b, d)
= (\eta\ot\eta)(\id\ot c^{-1}\ot\id)(a,b,c,d).
\end{align*}
\end{proof}

\begin{lema}\label{lemma-involution-2}
Let $\eta=\eta_1\otimes \eta_2$. Then 
\eqref{eq2} is equivalent to the following equations on $V^{\ot 4}$:
\begin{eqnarray}
(\eta\ot\eta)(\id\ot c\ot\id)(c\ot \id\ot\id) &=&  (\eta\ot\eta)(\id\ot c\ot\id)(\id\ot\id\ot c),\\
\eta_1\otimes \eta\otimes\eta_2 &=& (\eta\otimes\eta)(\id\otimes c_{312}).
\end{eqnarray}
Moreover, if $S_{H}^{2}=\id_{H}$ and $c^{2}=\id_{V\ot V}$, 
then \eqref{eq2} is always satisfied.
\end{lema}
\begin{proof}
The first equation is just a translation of (\ref{eq2}).  
The second equation is obtained directly from the first. 
The left hand side is obtained by invoking (\ref{eq5}); the right hand side by using 
$c_{312}=c\ot\id (\id\ot c)$. Now if $S_{H}$ and $c$ are involutions, 
then by Lemma \ref{lemma-involution-1} we have that
$$
[(\eta\ot\eta)(\id\ot c\ot\id)(c\ot \id\ot\id)](\id\ot\id\ot c) = (\eta\ot\eta)(\id\ot c\ot\id).
$$
On the other hand 
$$
[(\eta\ot\eta)(\id\ot c\ot\id)(\id\ot\id\ot c)](\id\ot\id\ot c)
=(\eta\ot\eta)(\id\ot c\ot \id)(\id\ot\id\ot c^2) = 
(\eta\ot\eta)(\id\ot c\ot \id).
$$
\end{proof}
The following corollary follows directly from the lemmas proved above.
\begin{cor}\label{cor:linear-2-cocycle}
Assume $H$ is a semisimple Hopf algebra and $c^{2}=\id_{V}$.
Let $\eta:V\ot V\to \Bbbk$ be an $H$-invariant linear functional. 
Then $\sigma=e^{\tilde{\eta}} \in \Hom_{\Bbbk}(A\ot A,\Bbbk) $ is a multiplicative
$2$-cocycle.\qed 
\end{cor}

\subsection{A closed formula for deformations of braided commutator relations}
We end this section with the following lemma that 
will be very useful in finding liftings. 
In particular, it will tell us how to deform relations like $[x,y]_{c}=0$, which are 
given by braided commutators.

\begin{lema}\label{lem:eta=sigma}
Let $\eta:V\ot V\to \Bbbk$ be an $H$-invariant linear functional
such that 
$\sigma = e^{\tilde{\eta}}$ is a multiplicative $2$-cocycle. 
Let $x_{1}, x_{2} \in V$ be homogeneous elements with
$\delta(x_{1}) = h_{1}\ot x_{1}$ and $\delta(x_{1}) = h_{2}\ot x_{2}$,
$h_{1}, h_{2}\in G(H)$, and
denote $z_{1}=x_{1}\# 1 $, $z_{2}= x_{2}\# 1 \in A$.
Then
$\sigma(z_{1}, z_{2}) =
\eta(x_{1}, x_{2})$. In particular, in $A_{\sigma}$ it holds that
\begin{align*}
z_{1}\cdot_{\sigma} z_{2} & = 
\eta(x_{1},x_{2})(1-h_{1}h_{2}) + z_{1}z_{2}. 
\end{align*}
\end{lema}

\pf
First we show that $\tilde{\eta}^{2}(z_{1}, z_{2}) = 0$.
Since $x_{1}$, $x_{2}$ are homogeneous we have
that 
$\com(z_{i}) = z_{i}\ot 1 + h_{i}\ot
z_{i}$, that is, $z_{i}$ is $(1,h_{i})$-primitive for $i=1,2$. Since 
by Lemma \ref{lem:funtional-2cocycle} we have 
$\tilde{\eta}|_{A_{0}\ot
A + A \ot A_{0}} = 0$, it follows that
\begin{align*}
& \tilde{\eta}^{2}(z_{1}, z_{2})  = 
 \tilde{\eta}([z_{1}]_{(1)}, [z_{2}]_{(1)}) 
 \tilde{\eta}([z_{1}]_{(2)}, [z_{2}]_{(2)}) = \\
& \quad = \tilde{\eta}(z_{1}, z_{2}) 
\tilde{\eta}(1, 1) + 
\tilde{\eta}(z_{1},h_{2} ) \tilde{\eta}(1,z_{2} ) +
\tilde{\eta}( h_{1}, z_{2} ) \tilde{\eta}(z_{1}, 1)
+
\tilde{\eta}( h_{1}, h_{2}) 
\tilde{\eta}(z_{1}, z_{2}) = 0.
\end{align*}
Thus, $\sigma(z_{1}, z_{2})
= \eps(z_{1})\eps(z_{2}) +
 \tilde{\eta}(z_{1}, z_{2}) = 
{\eta}(x_{1}, x_{2})$; in particular,
$\sigma^{-1}(z_{1}, z_{2}) $ $= e^{-\tilde{\eta}}(z_{1}, z_{2}) =  -
{\eta}(x_{1}, x_{2})$.
Finally,
\begin{align*}
 & z_{1}\cdot_{\sigma} z_{2}  =
\sigma([z_{1}]_{(1)}, [z_{2}]_{(1)}) 
[z_{1}]_{(2)} [z_{2}]_{(2)} 
\sigma^{-1}([z_{1}]_{(3)}, [z_{2}]_{(3)})\\
& \qquad = \sigma(z_{1}, z_{2})  \sigma^{-1}(1,1) + 
\sigma(z_{1},h_{2}) z_{2} \sigma^{-1}(1,1) +
\sigma(z_{1},h_{2}) h_{2} 
\sigma^{-1}(1,z_{2}) + \sigma(h_{1}, 
z_{2})z_{1}  \sigma^{-1}(1,1) +   \\
&  \qquad \qquad + 
\sigma(h_{1}, 
h_{2})z_{1} z_{2} \sigma^{-1}(1,1) + \sigma(h_{1}, 
h_{2})z_{1}h_{2} 
 \sigma^{-1}(1,z_{2})
+ 
\sigma(h_{1}, 
z_{2})h_{1} 
 \sigma^{-1}(z_{1},1) + 
\\
&  \qquad \qquad +
\sigma(h_{1}, 
h_{2})h_{1}z_{2} 
 \sigma^{-1}(z_{1},1) +
\sigma(h_{1}, 
h_{2})h_{1}h_{2} 
 \sigma^{-1}(z_{1},z_{2})=\\
& \qquad = \sigma(z_{1}, z_{2})  \sigma^{-1}(1,1) +
\sigma(h_{1}, 
h_{2})z_{1} z_{2} \sigma^{-1}(1,1)
+
\sigma(h_{1}, 
h_{2})h_{1}h_{2} 
 \sigma^{-1}(z_{1},z_{2})=\\
& \qquad = \eta(x_{1}, x_{2}) +
z_{1} z_{2} -
h_{1}h_{2} 
 \eta(x_{1},x_{2}) = \eta(x_{1},x_{2})(1 - h_{1}h_{2} ) + z_{1} z_{2},
\end{align*}
which finish the proof.
\epf

\section{On pointed Hopf algebras over dihedral groups}
\label{sec:pointed-dihedral}
All pointed Hopf algebras with group of group-likes isomorphic
to $ \dm $ with $m=4t \geq 12$ were classified in \cite{FG}.
To give the complete list we need first to
introduce some terminology. 

For the dihedral group $\dm$ we use the following 
presentation by generators and relations
\begin{equation}\label{eq:def-dihedral}
\mathbb D_m:=\langle g,h| \ g^2=1=h^m  \,\, , \,\,
gh=h^{-1}g \rangle.
\end{equation}

Because of our purposes we  
assume that $m=4t \geq 12$, $n=\frac{m}{2}= 2t$ and we fix 
$\omega$ an $ m $-th primitive root of unity.
Recall that the non-trivial conjugacy classes of $ \dm $ and the corresponding 
centralizers are 
\begin{itemize}
\item $\Oc_{h^{n}}= \{h^{n}\}$ and $C_{\dm}(h^{n})= \dm$.
\item $\Oc_{h^{i}}= \{h^{i}, h^{m-i}\}$ and $C_{\dm}(h^{i})= 
\langle
h \rangle \simeq  \Z/(m)$, for $1\leq i <n$. 
\item $\Oc_{g}= \{gh^{j}: j\text{ even}\}$ and $ C_{\dm}(g) = 
\langle g \rangle \times \langle h^{n}\rangle \simeq \Z/(2) \times \Z/(2) $.
\item$\Oc_{gh}= \{gh^{j}:j\text{ odd}\}$ and  
$ C_{\dm}(gh) = 
\langle gh \rangle \times \langle h^{n}\rangle \simeq \Z/(2) \times \Z/(2) $.
\end{itemize}

\subsection{Yetter-Drinfeld modules and Nichols
algebras over $ \dm $}\label{subsec:ydnadm}
The irreducible Yetter-Drinfeld modules that give rise
to finite dimensional Nichols algebras are associated to
the conjugacy classes of $h^{n}$ and $h^{i}$ with $1\leq i < n$, see
\cite[Table 2]{FG}. Next we describe them explicitly as well as
the families of reducible Yetter-Drinfeld modules with finite dimensional Nichols
algebras associated to them. Following a suggestion of the referee we slightly changed
the notation used in \cite{FG}.

\subsubsection{Yetter-Drinfeld modules 
and Nichols algebras associated $\oc_{h^n}$} 
Since
$h^n$ is central, $\oc_{h^n}=\{h^n\}$ and $C_{\dm}(h^n)=\dm$. 
The irreducible representations of $\dm$ are
well-known and they are of degree 1 or 2. Explicitly, there are:
\begin{itemize}
  \item[(i)] $n-1=\frac{m-2}{2}$ irreducible representations of degree 2 given by
$\rho_\ell:\dm\to\GL(2,\k)$ with
\begin{align}\label{eq:repdeg2}
\rho_\ell (g^ah^b)=\begin{pmatrix} 0 & 1 \\
                  1 & 0 \\
                \end{pmatrix}^a
 \begin{pmatrix}
                  \omega^\ell & 0 \\
                  0 & \omega^{-\ell} \\
                \end{pmatrix}^b, \quad \ell\in \N\text{ odd with }1\leq \ell<n.
\end{align}
\item[(ii)] 4 irreducible representations of degree 1.
They are given by the following table
\end{itemize}
\begin{table}[h]
\begin{center}
\begin{tabular}{|c|c|c|c|c|c|}
\hline $\sigma$& 1 & $h^n$ & $h^b$, $1\leq b\leq n-1$ & $g$ & $gh$

\\ \hline  \hline $\chi_1$ &1&1&1&1&1

\\ \hline $\chi_2$ &1&1&1&$-1$&$-1$

\\ \hline $\chi_3$ &1&$(-1)^n$& $(-1)^b$ &1&$-1$

\\ \hline $\chi_4$ &1&$(-1)^n$& $(-1)^b$ &$-1$&1\\

\hline
\end{tabular}
\end{center}
\end{table}

The irreducible Yetter-Drinfeld modules with finite dimensional Nichols algebra
are the ones given by the two-dimensional representations 
$\rho_{\ell}$ with $\ell\in \N$ odd.

Fix $\ell \in \N$ odd with $1\leq \ell <n$ and consider  
the two-dimensional simple representation $(\rho_{\ell}, V)$ of
$\dm$ described in \eqref{eq:repdeg2}
above. Let $\{v_{1}^{\ell},v_{2}^{\ell}\}$ be a $\Bbbk$-basis of $V$. Then
$M_{\ell} = M(\oc_{h^n},\rho_\ell) =  \k\dm\otimes_{\k\dm} V$ is the
Yetter-Drinfeld module spanned linearly by the elements 
$ x^{(\ell)}_{1} = 1\ot v_{1}^{\ell}, x^{(\ell)}_{2}=1\ot v_{2}^{\ell} $;
its structure is given by
\begin{align*}
g\cdot x^{(\ell)}_{1}  &= x^{(\ell)}_{2},  \quad & h\cdot x^{(\ell)}_{1}  
& = \omega^{\ell} x^{(\ell)}_{1},
\quad & \delta(x^{(\ell)}_{1}) &  =   h^{n}\ot x^{(\ell)}_{1},\\
\nonumber g\cdot x^{(\ell)}_{2}  & =  x^{(\ell)}_{1}, 
\quad & h\cdot x^{(\ell)}_{2}  & =  \omega^{-\ell} x^{(\ell)}_{2},
\quad & \delta(x^{(\ell)}_{2}) & = h^{n}\ot x^{(\ell)}_{2}.
\end{align*}
In particular, $\dim M_\ell =2$. By
 \cite[Thm. 3.1]{AF-alt-die}, one has that 
$\toba(\oc_{y^n}, \rho_{\ell}) \simeq \bigwedge
M_{\ell}$ and consequently
 $\dim \toba(\oc_{y^n},
\rho_{\ell})= 4$.

Consider now the set 
$\Lc$  of all sequences of finite length $(\ell_{1},\ldots,\ell_{r})$ with 
$ \ell_{i} \in \N$  odd  and 
$1\leq \ell_{1},\ldots, \ell_{r}<n$. 
Then for $L = (\ell_{1},\ldots,\ell_{r}) \in \Lc$ we define
$M_{L} = \bigoplus_{1\leq i \leq r} M_{\ell_{i}}$. 
Clearly, $M_{L} \in \yddm$ is \emph{reducible} and 
by \cite[Prop. 2.8]{FG}, we have that 
$\toba(M_{L})\simeq \bigwedge M_{L} $ and $\dim
\toba(M_{L})=4^{|L|}$, where $|L|=r$ denotes the length of $L$. 

\begin{obs}\label{rmk:cond-action-L}
Since $a\cdot x_{i}^{(\ell)} = \chi_{i}^{(\ell)}(a) x_{\sigma^{(\ell)}(a)(i)}$ for all
$a\in \dm$,
for these Nichols algebras
assumption 
\eqref{eq:assumption-yd} is satisfied with 
$\sigma^{(\ell)}(g) = (1 2)$, $\sigma^{(\ell)} (h) = \id$ and 
$\chi^{(\ell)}_{1}(g) = 1 =\chi^{(\ell)}_{2}(g)$, 
$\chi^{(\ell)}_{1}(h) = \omega^{\ell}$, 
$\chi^{(\ell)}_{2}(h) = \omega^{-\ell}$.
 \end{obs}

\subsubsection{Yetter-Drinfeld modules 
and Nichols algebras associated $\oc_{h^i}$}
Let $1\leq i <n$. In this case, $\oc_{h^i} = \{h^{i}, h^{m-i}\}$ and
$C_{\dm}(h^i)=\langle h \rangle \simeq \Z / (m) $. For $0\leq k <m$ denote by 
$\Bbbk_{\chi_{(k)}}$ the simple representation of $C_{\dm}(h^{i})$
given by the character $\chi_{(k)}(h) = \omega^{k}$.

Take $e$ and $g$ as representatives of left coclasses in $\dm / \langle h \rangle$,
with $h^{i} = eh^{i}e$ and $h^{m-i} = gh^{i}g$.
Then $M_{(i,k)} = M(\oc_{h^i},\chi_{(k)})= 
\k\dm\otimes_{\k \langle h \rangle} \Bbbk_{\chi_{(k)}}$  is the
Yetter-Drinfeld module spanned linearly by the elements 
$y^{(i,k)}_{1} = e\ot 1$ and $y^{(i,k)}_{2} = g\ot 1$.
Its structure is given by
\begin{align}\label{eq:ydik}
g\cdot y^{(i,k)}_{1} & =  y^{(i,k)}_{2},\quad & 
h\cdot y^{(i,k)}_{1}
& =  \omega^{k} y^{(i,k)}_{1},
\quad 
& \delta(y^{(i,k)}_{1}) & =   h^{i}\ot y^{(i,k)}_{1},\\
\nonumber
g\cdot y^{(i,k)}_{2}  & =  y^{(i,k)}_{1},
\quad & h\cdot y^{(i,k)}_{2}  & = \omega^{-k} y^{(i,k)}_{2}
,\quad
& \delta(y^{(i,k)}_{2}) &= h^{-i}\ot y^{(i,k)}_{2}.
\end{align}
In particular, $\dim M_{(i,k)} =2$. 

The irreducible Yetter-Drinfeld modules with finite dimensional Nichols algebra
are the ones given by the pairs $(i,k)$  satisfying that $\omega^{ik}=-1$.
We set $J=\{(i,k):\ 1\leq i< n, 1\leq k < m \text{ such that }\omega^{ik}=-1\}$.
By
 \cite[Thm. 3.1]{AF-alt-die}, one has that
$\toba(\oc_{y^i},  \chi_{(k)}) \simeq \bigwedge
M_{(i,k)}$, for all $(i,k)\in J$ and
 $\dim \toba(\oc_{y^i},
\chi_{(k)})= 4$. 

Consider now the set 
$\II$  of all sequences of finite length  of ordered pairs 
$((i_{1},k_{1}),\ldots,(i_{r},k_{r}))$ such that  $(i_{s},k_{s}) \in J$ 
and $\omega^{i_{s}k_{t}+ i_{t}k_{s}} = 1$ for 
all $1\leq s,t\leq r$.

For $I = ((i_{1},k_{1}),\ldots,(i_{r},k_{r})) \in \II$, we define $M_{I}=
\bigoplus_{1\leq j\leq r}M_{(i_{j},k_{j})}$.
By \cite[Prop. 2.5]{FG}, we have that
$\toba(M_{I})\simeq \bigwedge M_{I}$  and
$\dim\toba(M_{I}) = 4^{|I|}$, where $|I|=r$ denotes the length of $I$.

\begin{obs}\label{rmk:cond-action-I}
Since $a\cdot y_{j}^{(i,k)} = \chi_{j}^{(i,k)}(a)
y^{(i,k)}_{\sigma^{(i,k)}(a)(j)}$ for all
$a\in \dm$, for these Nichols algebras, 
assumption 
\eqref{eq:assumption-yd} is satisfied with 
$\sigma^{(i,k)}(g) = (1 2)$, $\sigma^{(i,k)} (h) = \id$ and 
$\chi^{(i,k)}_{1}(g) = 1 =\chi^{(i,k)}_{2}(g)$, 
$\chi^{(i,k)}_{1}(h) = \omega^{k}$, 
$\chi^{(i,k)}_{2}(h) = \omega^{-k}$.
 \end{obs}
\subsubsection{Yetter-Drinfeld modules 
and Nichols algebras associated to mixed classes}
Finally, we describe a family of reducible Yetter-Drinfeld modules
given by direct sums of the modules described above.

Let $\K$ be the set of all pairs of sequences of finite length 
$(I,L)$ with $I= ((i_{1},k_{1}),\ldots,(i_{r},k_{r})) \in \II$ and 
$L = (\ell_{1},\ldots,\ell_{s})\in \Lc$ such that $k_{j}$ is odd 
for all $1\leq j\leq r$ and $\omega^{i_{j}\ell_{t}}=-1$
for all $1\leq j\leq r$ and $1\leq t\leq s$.

As before, for $(I,L) \in \K$ we define $M_{I,L} =
\left(\bigoplus_{1\leq j\leq s} M_{(i_{j},k_{j})}\right)\oplus
\left(\bigoplus_{1\leq t\leq s} M_{\ell_{t}}\right) $. 
By \cite[Prop. 2.12]{FG}, we have that
$\toba(M_{I,L}) \simeq \bigwedge M_{I,L}$ and $\dim \toba(M_{I,L})
= 4^{\vert I \vert + |L|} $.

We end this subsection recalling the following classification result.
\begin{theorem}\cite[Thm. A]{FG}\label{thm:all-nichols-dm}
Let $\toba(M)$ be a finite dimensional Nichols algebra
in $ \yddm $. Then $\toba(M)\simeq \bigwedge M$, with $M$
isomorphic
either to $M_{I}$, or to
$M_{L}$, or to $M_{I,L}$, with $I\in \II$, $L\in \Lc$ and  
$(I,L) \in \K$, respectively.\qed
\end{theorem}

\subsection{Classification of finite dimensional 
pointed Hopf algebras over $ \dm $}
\label{subsec:phadm}
In this subsection we present all finite dimensional pointed Hopf algebras
over $ \dm $ up to isomorphism. First we introduce 
two families of quadratic algebras depending on 
families of parameters. It turns out that these quadratic
algebras give all nontrivial liftings of bosonizations of
finite dimensional 
Nichols algebras  in $ \yddm $.  

Let $I\in \II$ and $L\in \Lc$
be as Subsection \ref{subsec:ydnadm}. 
By abuse of notation, if $I =((i_{1},k_{1}),\ldots,(i_{r},k_{r}))$ 
and $L =(\ell_{1},\ldots,\ell_{s})$, we write $(i_{j},k_{j}) \in I$ and 
$\ell_{t} \in L$ for all $1\leq j\leq r$ and $1\leq t\leq s$.

Consider the families 
$\lambda = (\lambda_{p,q,i,k})_{(p,q),(i,k) \in I}$, $\gamma =
(\gamma_{p,q,i,k})_{(p,q),(i,k) \in I}$, $\theta =
(\theta_{p,q,\ell})_{(p,q) \in I, \ell \in L}$ and $\mu =
(\mu_{p,q,\ell})_{(p,q)\in I, \ell \in L}$ of elements
in $\k$ satisfying:
\begin{equation}\label{eq:cond-par-1}
\lambda_{p,m-k,i,k}=  \lambda_{i,k,p,m-k}\quad\text{and}\quad
\gamma_{p,k,i,k}= \gamma_{i,k,p,k}.
\end{equation}

\begin{definition}\label{def:AI}
For $I \in \II$, denote by
$A_{I}(\lambda, \gamma)$ the
algebra generated by $g,h, a^{(p,q)}_{1},a^{(p,q)}_{2}$
with $(p,q) \in I$ satisfying
the relations:
\begin{align*}
g^{2} = 1 = h^{m},\quad ghg= h^{m-1},\quad
ga^{(p,q)}_{1}  = a^{(p,q)}_{2} g,\quad  ha^{(p,q)}_{1} = 
\omega^{q} a^{(p,q)}_{1} h,
\quad ha^{(p,q)}_{2} = \omega^{-q}a^{(p,q)}_{1}h,\\
a^{(p,q)}_{1}a^{(i,k)}_{1} +a^{(i,k)}_{1}a^{(p,q)}_{1}  =
\delta_{q,m-k} \lambda_{p,q,i,k}(1 - h^{p+i}),
\
a^{(p,q)}_{1}a^{(i,k)}_{2} +a^{(i,k)}_{2}a^{(p,q)}_{1}  =
\delta_{q,k} \gamma_{p,q,i,k}(1 - h^{p-i}).
\end{align*}
\end{definition}
It is a Hopf algebra with its structure determined by $g,h$ being group-likes and
$$
\com(a^{(p,q)}_{1})  = a^{(p,q)}_{1}\ot 1 + h^{p}\ot a^{(p,q)}_{1},\qquad
\com(a^{(p,q)}_{2}) = a^{(p,q)}_{2}\ot 1 + h^{-p}\ot a^{(p,q)}_{2},
\text{ for all }(p,q)\in I.
$$
It turns out that the diagram of $A_{I}(\lambda, \gamma) $ is
exactly $\toba(M_{I}) $, thus 
we call the pair $
(\lambda,\gamma) $ a \emph{lifting datum} for $\toba(M_{I})$.
Set $\gamma = 0$ if $|I|=1$.

\begin{exa} If $ I = (i,n) $ with $ i $ odd
we obtain
the Hopf algebra $A_{(i,n)}(\lambda)$ generated
by the elements $g, h, a_{1}, a_{2}$ satisfying
\begin{eqnarray}
\nonumber g^{2} = 1 = h^{m},&\qquad ghg= h^{m-1},\\
\nonumber ga_{1}  = a_{2} g,&\qquad ha_{1} = -a_{1} h,
\qquad ha_{2} = -a_{2}h,\\
\nonumber a_{1}^{2}= \lambda (1 - h^{2i}),
&\quad a_{2}^{2}=
 \lambda(1 - h^{-2i}),\quad
a_{1}a_{2} + a_{2}a_{1} = 0.
\end{eqnarray}
\end{exa}

Now we introduce the second family of quadratic algebras.
\begin{definition}\label{def:liftings2}
For $(I,L) \in \K$, denote by
$B_{I,L}(\lambda,\gamma,\theta,\mu)$ the
algebra generated by  $g,h, a_{1}^{(p,q)}$, $a_{2}^{(p,q)},
b_{1}^{(\ell)},b_{2}^{(\ell)}$
satisfying
the relations:
\begin{align*}
g^{2} = 1 = h^{m},\quad ghg= h^{m-1},\quad
g a_{1}^{(p,q)}  =  a_{2}^{(p,q)} g,\qquad\qquad\qquad\qquad\qquad\\
h a_{1}^{(p,q)} = \omega^{q}  a_{1}^{(p,q)} h,
\quad
gb_{1}^{(\ell)}  = b_{2}^{(\ell)} g,\qquad
hb_{1}^{(\ell)} = \omega^{\ell} b_{1}^{(\ell)} h,
\qquad\qquad\qquad\qquad\qquad\\
 [a_{1}^{(p,q)}]^{2} = 0 =  [a_{2}^{(p,q)}]^{2},\qquad
b_{1}^{(\ell)}b_{2}^{(\ell')}+b_{2}^{(\ell')}b_{1}^{(\ell)}= 0,\qquad
b_{1}^{(\ell)}b_{1}^{(\ell')}+b_{1}^{(\ell')}b_{1}^{(\ell)}= 0,\qquad\qquad\\
 a_{1}^{(p,q)}a_{1}^{(i,k)} +a_{1}^{(i,k)}a_{1}^{(p,q)} =
\delta_{q,m-k} \lambda_{p,q,i,k}(1 - h^{p+i}),\
 a_{1}^{(p,q)} a_{2}^{(i,k)} +a_{2}^{(i,k)} a_{1}^{(p,q)} =
 \delta_{q,k}\gamma_{p,q,i,k}(1 - h^{p-i}),\\
 a_{1}^{(p,q)}b_{1}^{(\ell)} + b_{1}^{(\ell)} a_{1}^{(p,q)}=
 \delta_{q,m-\ell}\theta_{p,q,\ell}(1 - h^{n+p}),\
a_{1}^{(p,q)}b_{2}^{(\ell)} + b_{2}^{(\ell)} a_{1}^{(p,q)} =
 \delta_{q,\ell}\mu_{p,q,\ell}(1 - h^{n+p}).
\end{align*}
\end{definition}
It is a Hopf algebra with its structure determined by $g,h$ being group-likes and 
\begin{align*}
\com(a_{1}^{(p,q)}) & = a_{1}^{(p,q)}\ot 1 + h^{p}\ot a_{1}^{(p,q)},&
\com(a_{2}^{(p,q)}) & =a_{2}^{(p,q)}\ot 1 + h^{-p}\ot a_{2}^{(p,q)},\\
\com(b_{1}^{(\ell)}) & = b_{1}^{(\ell)}\ot 1 + h^{n}\ot b_{1}^{(\ell)},&
\com(b_{2}^{(\ell)})& = b_{2}^{(\ell)}\ot 1 + h^{n}\ot b_{2}^{(\ell)},
\end{align*} 
for all $ (p,q)\in I, \ell \in L $.
It turns out that the diagram of $B_{I,L}(\lambda,\gamma,\theta,\mu)$ is
$\toba(M_{I,L})$, thus 
we call the 4-tuple $
(\lambda,\gamma,\theta,\mu) $ a lifting datum for
$\toba(M_{I,L})$.

\begin{exa} Let $I=\{(i,k)\}$ and $L=\{m-k\}$
 with $1\leq k<m$ an odd number. The
Hopf algebra $B_{I,L}(\theta,\mu)$ is the algebra generated
by $g, h, a_{1}, a_{2},b_{1},b_{2}$ satisfying the
relations
\begin{eqnarray}\nonumber
g^{2} = 1 = h^{m},&\qquad ghg= h^{m-1},\\\nonumber
ga_{1}  = a_{2} g,&\qquad ha_{1} = \omega^{k}a_{1} h,
\qquad
gb_{1}  = b_{2} g,  \qquad hb_{1} =  \omega^{-k}b_{1}h,
\\\nonumber
a_{1}^{2}= 0 = a_{2}^{2},&\qquad b_{1}^{2} = 0 = b_{2}^{2},\qquad
a_{1}a_{2} + a_{2}a_{1} = 0,\\\nonumber
b_{1}b_{2} +b_{2}b_{1}= 0,& 
a_{1}b_{1}+b_{1}a_{1}= \theta (1 - h^{n+i}),\quad
a_{1}b_{2}+b_{2}a_{1}= \mu(1 - h^{n+i}).\nonumber
\end{eqnarray}
\end{exa}

The following theorem gives the classification of 
all finite dimensional 
pointed Hopf algebras over $ \dm $ with $ m=4t\geq 12 $.

\begin{theorem}\cite[Thm. B]{FG}
Let $H$ be a finite dimensional pointed Hopf algebra with $G(H) = \dm$,
$ m=4t\geq 12 $.
Then $H$ is isomorphic to one of the following:
\begin{enumerate}
\item[$ (a) $] $\toba(M_I)\#\Bbbk \dm$ with $I= ((i,k)) \in \II$, $k\neq n$, or
\item[$ (b) $] $\toba(M_{L})\#\Bbbk \dm$ with $L\in \Lc$, or
\item[$ (c) $] $A_{I}(\lambda,\gamma)$ with $I\in \II$, $|I| > 1$
or $I = ((i,n))$ and $\gamma \equiv 0$, or
\item[$ (d) $] $B_{I,L }(\lambda,\gamma,\theta,\mu)$ with 
$(I,L)\in \K$, $|I|> 0$ and $|L|>0$.
\end{enumerate}
Conversely, any pointed Hopf algebra of the list above is a lifting
of a boso\-nization of a finite dimensional braided Hopf algebra in $ \yddm $.\qed
\end{theorem}

\subsection{Cocycle deformations and 
finite dimensional pointed Hopf algebras over $ \dm $}
In this subsection we prove that all pointed Hopf algebras
$A_{I}(\lambda,\gamma)$ and $B_{I,L}(\lambda, \gamma,
\theta, \mu) $ can be obtained by deforming
the multiplication of a bosonization of a Nichols algebra using
a multiplicative $2$-cocycle. 

\subsubsection{Cocycle deformations 
and the algebras $A_{I}(\lambda,\gamma)$}
Let $ I\in \II $ and consider the Nichols algebra $ \toba(M_{I}) $.
For all $ (p,q) \in I $ on $ M_{I}$, consider the linear maps 
$ d_{1}^{(p,q)}, d_{2}^{(p,q)}$ given
by the rule $ d_{r}^{(p,q)}(y_{s}^{(i,k)}) = \delta_{r,s}\delta_{p,i}\delta_{q,k}$ 
for all $r,s=1,2$, $(p,q), (i,k) \in I $. By Subsection \ref{subsec:hoch-H-inv} the
following map defines a Hochschild $2$-cocycle on $\toba(M_{I})$
$$ \eta = \sum_{\substack{(p,q),(i,k) \in I,\\ 1\leq r,s\leq 2}} \alpha_{p,q, i,k}^{r,s}
d_{r}^{(p,q)} \ot d_{s}^{(i,k)}.$$ 

\begin{lema}\label{lem:inv-bider-I}
 $ \eta $ is $ \dm $-invariant if and only if the following conditions hold:
\begin{eqnarray}
\alpha_{p,q, i,k}^{r,s}  =  & \alpha_{p,q, i,k}^{s,r} & \forall 
(p,q),(i,k) \in I, r,s=1, 2,\label{eq:inv-bider-I1}\\
\alpha_{p,q, i,k}^{1,1}  =  & \alpha_{p,q, i,k}^{2,2} & \forall 
(p,q),(i,k) \in I, \label{eq:inv-bider-I2}\\
\alpha_{p,q, i,k}^{r,r}  = & \delta_{q,m-k}\alpha_{p,m-k, i,k}^{r,r} & \forall 
(p,q),(i,k) \in I, r=1, 2,\label{eq:inv-bider-I3}\\
\alpha_{p,q, i,k}^{r,s}  = & \delta_{q,k}\alpha_{p,k, i,k}^{r,s} & \forall 
(p,q),(i,k) \in I, 1\leq r\neq s\leq 2. \label{eq:inv-bider-I4}
\end{eqnarray}
 \end{lema}

\pf
To prove that $ \eta $ is $ \dm $-invariant it is enough to 
show that $ \eta^{g} = \eta^{h}=\eta  $. Since 
$ [d_{1}^{(p,q)}]^{g} = d_{2}^{(p,q)} $ and 
$ [d_{2}^{(p,q)}]^{g} = d_{1}^{(p,q)} $ for all 
$ (p,q) \in I $, and $ \eta $ is a linear combination 
of tensor products of $\ep$-derivations, we have that 
$\eta^{g} = \eta$ if and only if \eqref{eq:inv-bider-I1} and
\eqref{eq:inv-bider-I2} hold. Analogously, since 
$  [d_{i}^{(p,q)}]^{h} = \omega^{(-1)^{i-1}q} d_{i}^{(p,q)} $ for all 
$ (p,q) \in I $ and $ i=1,2 $ we have that $ \eta^{h} = \eta $ if and only if 
$$\eta = \sum_{\substack{(p,q),(i,k) \in I,\\ 1\leq r,s\leq 2}} \alpha_{p,q, i,k}^{r,s}
\omega^{(-1)^{r-1}q + (-1)^{s-1}k} d_{r}^{(p,q)} \ot d_{s}^{(i,k)},$$
which holds if and only if 
$ \alpha_{p,q, i,k}^{r,s} = \alpha_{p,q, i,k}^{r,s}
\omega^{(-1)^{r-1}q + (-1)^{s-1}k} $
for all $ (p,q)$, $(i,k)\in I $ and $ r,s=1,2 $.
Thus, if $ r=s $ we must have that
$\alpha_{p,q, i,k}^{r,r} = 0 $ or
 $ q  \equiv -k \mod m $ which gives \eqref{eq:inv-bider-I3} and
 if $ r\neq s $ then 
 $\alpha_{p,q, i,k}^{r,s} = 0 $ or
 $ q  \equiv k \mod m $ which gives \eqref{eq:inv-bider-I4}.  
\epf

\begin{lema}\label{lem:AI-cocycle}
Assume $ \eta $ satisfies conditions 
\eqref{eq:inv-bider-I1} -- \eqref{eq:inv-bider-I4}.
Then 
$ \sigma = e^{\tilde{\eta}}  $ is a 
multiplicative $ 2 $-cocycle for $ \toba(M_{I})\# \Bbbk \dm $.
\end{lema}

\pf By assumption, we know that $\eta$
is $\dm$-invariant. Since by Theorem \ref{thm:all-nichols-dm}, 
the braiding in $\yddm$ is symmetric,
then by Lemmas \ref{lem:cond-eps-biderivations},
\ref{lemma-involution-1} and \ref{lemma-involution-2}, 
we get that $\tilde{\eta}$ fulfills the conditions
in Lemma \ref{le:cond-comm-cocycle}, and consequently 
$ \sigma = e^{\tilde{\eta}}  $ is a 
multiplicative $ 2 $-cocycle for $ \toba(M_{I})\# \Bbbk \dm $.
\epf

\begin{theorem}\label{thm:AI-cocycle}
Let $ H = \toba(M_{I})\# \Bbbk \dm $ and $\sigma = e^{\tilde{\eta}}$
be the multiplicative $2-cocycle$ given by \emph{Lemma \ref{lem:AI-cocycle}}. Then 
$ H_{\sigma} \simeq A_{I}(\lambda, \gamma)$ with 
$\lambda_{p,q,i,k}=\alpha^{r,r}_{p,q,i,k}+\alpha^{r,r}_{i,k,p,q}$
and $\gamma_{p,q,i,k}=\alpha^{r,s}_{p,q,i,k}+\alpha^{r,s}_{i,k,p,q}$
for all $(p,q),(i,k)\in I$. In particular, $A_{I}(\lambda, \gamma)$ is 
a cocycle deformation of $H$ for all lifting datum. 
\end{theorem}

\pf 
To show that $ H_{\sigma}$ is isomorphic to $A_{I}(\lambda, \gamma)$
it suffices to prove that the generators of $ H_{\sigma} $ satisfy the 
relations given in Definition \ref{def:AI}, for this would imply 
that there exists a Hopf algebra surjection 
$ H_{\sigma} \twoheadrightarrow A_{I}$ and since both
algebras have the same dimension they must be isomorphic.

For $(p,q)\in I$ and $1\leq r\leq 2$, denote $ a_{r}^{(p,q)}= 
y_{r}^{(p,q)}\# 1 \in \toba(M_{I})\# \k \dm$.
Then by Lemma 
\ref{lem:eta=sigma} we have  
for all $(p,q),(i,k) \in I$ and $r,s=1,2$ that
\begin{align*}
a_{r}^{(p,q)}\cdot_{\sigma} a_{s}^{(i,k)}  
& = \eta(y_{r}^{(p,q)}, y_{s}^{(i,k)})
(1 - h^{p(-1)^{r-1}}h^{i(-1)^{s-1}}) +
a_{r}^{(p,q)} a_{s}^{(i,k)} 
\\
  &= \alpha^{r,s}_{p,q,i,k}(1- h^{p(-1)^{r-1}+i(-1)^{s-1}}) +
a_{r}^{(p,q)} a_{s}^{(i,k)}.
\end{align*}
Using Lemma \ref{lem:inv-bider-I} we obtain that
\begin{align*}
a_{1}^{(p,q)}\cdot_{\sigma} a_{1}^{(i,k)} +
a_{1}^{(i,k)}\cdot_{\sigma} a_{1}^{(p,q)}  & = 
a_{1}^{(p,q)} a_{1}^{(i,k)} +  a_{1}^{(i,k)}a_{1}^{(p,q)} +
\delta_{q,m-k}(\alpha^{1,1}_{p,q,i,k} + 
\alpha^{1,1}_{i,k,p,q})(1- h^{p+i})\\
& = 
\delta_{q,m-k}(\alpha^{1,1}_{p,q,i,k} + 
\alpha^{1,1}_{i,k,p,q})(1- h^{p+i})\qquad\text{ and }\\
a_{1}^{(p,q)}\cdot_{\sigma} a_{2}^{(i,k)} +
a_{2}^{(i,k)}\cdot_{\sigma} a_{1}^{(p,q)}  & = 
a_{1}^{(p,q)} a_{2}^{(i,k)} +  a_{2}^{(i,k)}a_{1}^{(p,q)} +
\delta_{q,k}(\alpha^{1,2}_{p,q,i,k} + 
\alpha^{2,1}_{i,k,p,q})(1- h^{p-i})\\
& = 
\delta_{q,k}(\alpha^{1,2}_{p,q,i,k} + 
\alpha^{1,2}_{i,k,p,q})(1- h^{p-i}).
\end{align*}
Thus, defining $\lambda_{p,q,i,k} = \alpha^{r,r}_{p,q,i,k} + 
\alpha^{r,r}_{i,k,p,q}$ and 
$\gamma_{p,q,i,k} = \alpha^{r,s}_{p,q,i,k} + 
\alpha^{r,s}_{i,k,p,q}$ with $1\leq r\neq s \leq 2$ we 
get that condition \eqref{eq:cond-par-1} is satisfied. 
Since the other relations follows from the Yetter-Drinfeld
structure of $M_{I}$, the theorem is proved.
\epf

\begin{obs} Note that given a lifting datum $(\lambda,\gamma)$, 
using Lemma \ref{lem:inv-bider-I} and
Theorem \ref{thm:AI-cocycle} one is able to construct a multiplicative 
$2$-cocycle that gives the desired 
deformation of $\toba(M_{I})\#\Bbbk \dm$.
\end{obs}

\subsubsection{Cocycle deformations 
and the algebras $B_{I,L}(\lambda,\gamma, \theta, \mu)$}
Let $ (I,L) \in \K $ and 
consider the Nichols algebra $ \toba(M_{I, L}) $. 
For all $ (p,q) \in I, \ell \in L $, consider the linear maps  
$ d_{1}^{(p,q)}, d_{2}^{(p,q)}$
and $ d_{1}^{(\ell)}, d_{2}^{(\ell)} $ on $M_{I, L}$
given by the rules 
\begin{align*}
d_{r}^{(p,q)}(y_{s}^{(i,k)})  = \delta_{r,s}\delta_{p,i}\delta_{q,k},\quad
d_{r}^{(p,q)}(x_{s}^{(\ell)}) = 0,\quad
d_{r}^{(\ell)}(x_{s}^{(\ell')}) & = \delta_{r,s}\delta_{\ell,\ell'},\quad
d_{r}^{(\ell)}(y_{s}^{(i,k)})  = 0.
\end{align*}
for all $r,s=1,2$, $(p,q), (i,k) \in I,\ \ell \in L $. By Subsection \ref{subsec:hoch-H-inv} the
following map defines a Hochschild $2$-cocycle on $ \toba(M_{I, L}) $
\begin{align*}
\eta & = \sum_{\substack{(p,q),(i,k) \in I,\\ 1\leq r,s\leq 2}} \alpha_{p,q, i,k}^{r,s}
d_{r}^{(p,q)} \ot d_{s}^{(i,k)} + \sum_{\substack{(p,q) \in I,\ell \in L \\ 1\leq r,s\leq 2}} 
[\beta_{p,q, \ell}^{r,s}
d_{r}^{(p,q)} \ot d_{s}^{(\ell)} + \zeta_{p,q, \ell}^{r,s}
d_{s}^{(\ell)} \ot d_{r}^{(p,q)}] + \\
& +   
\sum_{\substack{\ell,\ell' \in L \\ 1\leq r,s\leq 2}} \xi_{\ell, \ell'}^{r,s}
d_{r}^{(\ell)} \ot d_{s}^{(\ell')}.
\end{align*}
  
\begin{lema}\label{lem:inv-bider-I-L}
 $ \eta $ is $ \dm $-invariant if and only if the following conditions hold:
 \eqref{eq:inv-bider-I1}--\eqref{eq:inv-bider-I4} 
 from Lemma \ref{lem:inv-bider-I},  
\begin{eqnarray}
\beta_{p,q,\ell}^{r,s}  =  & \beta_{p,q, \ell}^{s,r} & \forall 
(p,q)\in I,\ell \in L, r,s=1, 2,\label{eq:inv-bider-I-L-5}\\
\beta_{p,q,\ell}^{1,1}  =  & \beta_{p,q, \ell}^{2,2} & \forall 
(p,q)\in I,\ell \in L,\label{eq:inv-bider-I-L-6}\\
\beta_{p,q,\ell}^{r,r}  =  & \delta_{q,m-\ell}\beta_{p,m-\ell, \ell}^{r,r} & \forall 
(p,q)\in I,\ell \in L, r=1, 2,\label{eq:inv-bider-I-L-7}\\
\beta_{p,q,\ell}^{r,s}  =  & \delta_{q,\ell}\beta_{p,\ell, \ell}^{r,s} & \forall 
(p,q)\in I,\ell \in L, 1\leq r\neq s\leq 2,\label{eq:inv-bider-I-L-8}\\
\xi_{\ell,\ell'}^{r,s}  =  & \xi_{\ell,\ell'}^{s,r} & \forall 
\ell,\ell'\in L, r,s=1, 2,\label{eq:inv-bider-I-L-13}\\
\xi_{\ell,\ell'}^{r,r}  =  &0 & \forall 
\ell,\ell'\in L, r=1,2,\label{eq:inv-bider-I-L-14}\\
\xi_{\ell,\ell'}^{r,s}  = & \delta_{\ell,\ell'}\xi_{\ell,\ell'}^{r,s} & \forall 
\ell,\ell'\in L, 1\leq r\neq s\leq 2,\label{eq:inv-bider-I-L-15}
\end{eqnarray}
and the coefficients $\zeta^{r,s}_{p,q,\ell}$ satisfy the same 
conditions as the coefficients $\beta^{r,s}_{p,q,\ell}$, for all
$(p,q)\in I,\ \ell \in L,\ r,s=1,2$.
 \end{lema}

\begin{obs}\label{rmk:17-zero}
Note that in this case, equation \eqref{eq:inv-bider-I3} implies that 
$\alpha_{p,q, p,q}^{r,r} = 0$
for all $(p,q), (i,k) \in I$, since $m=4t$, $q$ is odd for all $(p,q) \in I, (I,L) \in \K$
and $m-q \equiv q \mod m$ if and only if $m=2q$.
\end{obs}

\pf
To prove that $ \eta $ is $ \dm $-invariant it is enough to 
show that $ \eta^{g} = \eta^{h}=\eta  $. Thus
the first four conditions follows directly from Lemma \ref{lem:inv-bider-I}.
The proof of the remaining conditions goes along the same lines. Only note 
that condition \eqref{eq:inv-bider-I-L-14} is different because 
it never holds that $ \ell' \equiv m -\ell \mod m$ since $ 1\leq \ell, \ell' <n $ and
$ m=2n $.
\epf

The proof of the following lemma is completely analogous to 
the proof of Lemma \ref{lem:AI-cocycle}.

\begin{lema}\label{lem:BIL-cocycle}
Assume $ \eta $ satisfies conditions 
\eqref{eq:inv-bider-I1} -- \eqref{eq:inv-bider-I-L-15}.
Then 
$ \sigma = e^{\tilde{\eta}}  $ is a 
multiplicative $ 2 $-cocycle for $ \toba(M_{I,L})\# \Bbbk \dm $.\qed
\end{lema} 

\begin{theorem}\label{thm:BIL-cocycle}
Let $ H = \toba(M_{I,L})\# \Bbbk \dm $ and $\sigma = e^{\tilde{\eta}}$
be the multiplicative $2-cocycle$ given by \emph{Lemma \ref{lem:BIL-cocycle}}. Then 
$ H_{\sigma}\simeq B_{I,L}(\lambda, \gamma,\theta,\mu)$ with 
$\lambda_{p,q,i,k}=\alpha^{r,r}_{p,q,i,k}+\alpha^{r,r}_{i,k,p,q}$,
$\gamma_{p,q,i,k}=\alpha^{r,s}_{p,q,i,k}+\alpha^{r,s}_{i,k,p,q}$,
$\theta_{p,q,\ell}=\beta^{1,1}_{p,q,\ell}+\zeta^{1,1}_{p,q,\ell}$,
and $\mu_{p,q,\ell}=\beta^{1,2}_{p,q,\ell}+\zeta^{1,2}_{p,q,\ell}$,
for all $(p,q)\in I, \ell \in L$. In particular, $B_{I,L}(\lambda, \gamma,\theta,\mu)$ is 
a cocycle deformation of $H$ for all lifting datum.
\end{theorem}

\pf
As in the proof of Theorem \ref{thm:AI-cocycle},
it suffices to show that the generators of $ H_{\sigma} $ satisfy the 
relations given in Definition \ref{def:liftings2}.
For $(p,q)\in I, \ell \in L$ and $1\leq r\leq 2$, denote $ a_{r}^{(p,q)}= 
y_{r}^{(p,q)}\# 1$ and $b_{r}^{(\ell)} = x_{r}^{(\ell)}\#1 \in 
\toba(M_{I,L})\# \k \dm$.

Since $\tilde{\eta}$ coincides with the multiplicative
cocycle given by Lemma \ref{lem:AI-cocycle} when it
takes values in $\{a_{r}^{(p,q)}:\ (p,q)\in I, r=1,2\}$, 
by the proof of Theorem \ref{thm:AI-cocycle} we have that the 
equations involving the generators $a_{r}^{(p,q)}$ are satisfied. 
In particular, since $q$ is odd for all $(p,q)$ we have that 
$q\not\equiv m-q \mod m$ for all $(p,q) \in I$ and 
by Lemma \ref{lem:eta=sigma} 
$$a_{r}^{(p,q)}\cdot_{\sigma} a_{r}^{(p,q)} = 
 [a_{r}^{(p,q)}]^{2} +
\delta_{q,m-q}\alpha^{r,r}_{p,q,p,q}(1- h^{2p(-1)^{r-1}})= 0.$$
Moreover, again by Lemma \ref{lem:eta=sigma}
we get that  
$$
 b_{r}^{(\ell)}\cdot_{\sigma} b_{s}^{(\ell') }
= \eta(x_{r}^{(\ell)}, x_{s}^{(\ell') })( 1 - h^{n}h^{n} )+
b_{r}^{(\ell)} b_{s}^{(\ell') } = b_{r}^{(\ell)} b_{s}^{(\ell') }
\quad\text{ for all }\ell, \ell' \in L, r,s=1,2.
$$
 
Hence, using the relations of the Nichols algebra $\toba(M_{I,L})$
we have that
$$  b_{r}^{(\ell)}\cdot_{\sigma} b_{s}^{(\ell') } + 
 b_{s}^{(\ell') }\cdot_{\sigma}
b_{r}^{(\ell)} = b_{r}^{(\ell)}b_{s}^{(\ell') } + 
 b_{s}^{(\ell') }
b_{r}^{(\ell)} = 0
\quad\text{ for all  }\ell, \ell' \in L, r,s=1,2. 
$$
Besides, by \eqref{eq:inv-bider-I-L-7} we get
\begin{align*}
 a_{1}^{(p,q)}\cdot_{\sigma} b_{1}^{(\ell) } 
& = \eta(y_{1}^{(p,q)}, x_{1}^{(\ell) }) (1 - h^{p}h^{n} ) +
a_{1}^{(p,q)} b_{1}^{(\ell) } 
= \delta_{q,m-\ell}\beta^{1,1}_{p,q, \ell }(1 - h^{p+n})+
 a_{1}^{(p,q)} b_{1}^{(\ell) }\qquad \text{ and }\\
 b_{1}^{(\ell) }\cdot_{\sigma} a_{1}^{(p,q)} 
& = \eta(x_{1}^{(\ell) }, y_{1}^{(p,q)}) (1 - h^{n}h^{p} ) +
b_{1}^{(\ell) }a_{1}^{(p,q)}  
= \delta_{q,m-\ell}\zeta^{1,1}_{p,q, \ell }(1 - h^{p+n})+
 b_{1}^{(\ell) } a_{1}^{(p,q)},
\end{align*}
for all $(p,q)\in I, \ell \in L$.
Hence, using again the relations of the Nichols algebra $\toba(M_{I,L})$
we have
$$ a_{1}^{(p,q)}\cdot_{\sigma} b_{1}^{(\ell) } + 
  b_{1}^{(\ell) }\cdot_{\sigma} a_{1}^{(p,q)}  = 
\delta_{q,m-\ell}(\beta^{1,1}_{p,q, \ell } + \zeta^{1,1}_{p,q, \ell })(1 - h^{p+n}).
$$
If we set $\theta_{p,q,\ell} = \beta^{1,1}_{p,q,\ell} + 
\zeta^{1,1}_{p,q,\ell}$ with $(p,q)\in I, \ell \in L$, then
the condition involving the generators $a_{1}^{(p,q)}, b_{1}^{(\ell) }$
is satisfied.
Finally, by \eqref{eq:inv-bider-I-L-8} we have that
\begin{align*}
 a_{1}^{(p,q)}\cdot_{\sigma} b_{2}^{(\ell) } & 
= \eta(y_{1}^{(p,q)}, x_{2}^{(\ell) }) (1 - h^{p}h^{n} ) +
a_{1}^{(p,q)} b_{2}^{(\ell) }
= \delta_{q,\ell}\beta^{1,2}_{p,q, \ell }(1 - h^{p+n})+
 a_{1}^{(p,q)} b_{2}^{(\ell) }\qquad \text{ and }\\
 b_{2}^{(\ell) }\cdot_{\sigma} a_{1}^{(p,q)} & 
= \eta(x_{2}^{(\ell) }, y_{1}^{(p,q)}) (1 - h^{n}h^{p} ) +
b_{2}^{(\ell) }a_{1}^{(p,q)}
= \delta_{q,\ell}\zeta^{1,2}_{p,q, \ell }(1 - h^{p+n})+
 b_{2}^{(\ell) } a_{1}^{(p,q)},
\end{align*}
for all $(p,q)\in I, \ell \in L$.
Thus
$$ a_{1}^{(p,q)}\cdot_{\sigma} b_{2}^{(\ell) } + 
  b_{2}^{(\ell) }\cdot_{\sigma} a_{1}^{(p,q)}  = 
\delta_{q,\ell}(\beta^{1,2}_{p,q, \ell } + \zeta^{1,2}_{p,q, \ell })(1 - h^{p+n}).
$$
Defining $\mu_{p,q,\ell} = \beta^{1,2}_{p,q,\ell} + 
\zeta^{1,2}_{p,q,\ell}$ with $(p,q)\in I, \ell \in L$, it follows that
the condition involving the generators $a_{1}^{(p,q)}, b_{2}^{(\ell) }$
is satisfied.
Since the other relations follows from the Yetter-Drinfeld
structure of $M_{I,L}$, the theorem is proved.
\epf

\begin{obs} Note that given a lifting datum $(\lambda,\gamma, \theta, \mu)$, using 
Lemma \ref{lem:inv-bider-I-L} and
Theorem \ref{thm:BIL-cocycle} one is able to construct a 
multiplicative $2$-cocycle that give the desired 
deformation.
\end{obs}

\section{On pointed Hopf algebras over symmetric groups}\label{sec:pointed-symm}
Finite dimensional pointed Hopf algebras whose coradical is the
group algebra of the groups $\s_3$ and $\s_4$ were classified in
\cite{AHS} and \cite{GG}, respectively.
In this section, we prove that some of them are cocycle deformations
by giving, as in Section \ref{subsec:phadm}, explicitly the cocycles.

\subsection{Racks, Yetter-Drinfeld modules and Nichols
algebras over $ \sbb_{n} $}\label{subsec:ydnasn}
To present finite dimensional Nichols algebras over $\s_{n}$
we need first to introduce the notion of racks, see \cite[Def. 1.1]{AG2} for 
more details.

A \emph{rack} is a pair $(X,\rhd)$,
where $X$ is a
non-empty set and $\rhd:X\times X\to X$ is a function, such that
$\phi_i=i\rhd (\cdot):X\to X$ is a bijection for all $i\in X$
satisfying that
$
i\rhd(j\rhd k)=(i\rhd j)\rhd (i\rhd k) \text{ for all }i,j,k\in X
$. A group $G$ is a rack with $x \trid y = xyx^{-1}$ for all 
$x, y \in G$. If $G=\s_{n}$, then we denote by $\mO_j^n$
the conjugacy class of all $j$-cycles in
$\s_n$. 

Let $(X,\rhd)$ be a rack. A
\textit{rack 2-cocycle} $q:X\times X\to \Bbbk^{\times}$,
$(i,j)\mapsto q_{ij}$ is a function such that 
$ q_{i,j\rhd k}\, q_{j,k}=q_{i\rhd j,i\rhd
k}\, q_{i,k}$, for all $\,i,j,k\in X$.
It determines a braiding $c^q$ on the
vector space $\Bbbk X$ with basis $\{x_i\}_{i\in X}$ by $c^q(x_i\otimes
x_j)=q_{ij}x_{i\rhd j}\otimes x_i$ for all $i,j\in X$. We denote by
$\toba(X,q)$ the Nichols algebra of this braided vector space $(\k X, c^{q})$. 

Let $X=\mO_2^n$ with $ n\geq 3 $ or $X=\mO_4^4$ and consider the cocycles:
\begin{align*}
 -1&:X\times X\to \Bbbk^{\times}, 
&& (j,i)\mapsto \sg(j)=-1, && i,j\in X;\\
 \chi&: \mO_2^n\times \mO_2^n\to 
\Bbbk^{\times}, && (j,i)\mapsto \chi_i(j) = \begin{cases}
  1,  & \mbox{if} \  i=(a,b) \text{ and } j(a)<j(b), \\
  -1, & \mbox{if} \ i=(a,b) \text{ and } j(a)>j(b).
\end{cases}
&& i,j\in\mO_2^n.
\end{align*}
By \cite[Ex. 6.4]{MS}, \cite{Gr2}, 
\cite[Thm. 6.12]{AG3},
\cite[Prop. 2.5]{GG}, the Nichols algebras are
given by 
\begin{enumerate}
 \item[$(a)$] ${\toba}(\mO_2^n,-1)$; generated
by the elements $\{x_{(\ell m)}\}_{ 1\le \ell < m \le n}$ satisfying
for all $1\le a < b < c \le n, 1\le e < f \le n, \{a,b\}\cap\{e,f\}=\emptyset$
that 
\begin{align*}
0= x_{(ab)}^2  = x_{(ab)}x_{(ef)}+x_{(ef)} x_{(ab)} = 
x_{(ab)}x_{(bc)}+x_{(bc)} x_{(ac)}+x_{(ac)} x_{(ab)}.
\end{align*}
 \item[$(b)$] ${\toba}(\mO_2^n,\chi)$; 
generated by the elements $\{x_{(\ell m)}\}_{ 1\le \ell < m \le n}$ satisfying
for all $1\le a < b < c \le n, 1\le e < f \le n, \{a,b\}\cap\{e,f\}=\emptyset$ that
\begin{align*}
0 &= x_{(ab)}^2 =  x_{(ab)}x_{(ef)} - x_{(ef)} x_{(ab)}
= x_{(ab)}x_{(bc)} - x_{(bc)} x_{(ac)} - x_{(ac)} x_{(ab)},\\
0&= 
x_{(bc)}x_{(ab)} -  x_{(ac)}x_{(bc)} - x_{(ab)}x_{(ac)}.
\end{align*}
 \item[$(c)$]  ${\toba}(\mO_4^4,-1) $; 
generated by the elements $x_i, i\in \mO_4^4$ satisfying for 
$ij=ki$ and $ j\neq i\neq k\in
\mO_4^4$ that  
\begin{align*}
0= x_i^2=  x_i
x_{i^{-1}}+x_{i^{-1}} x_i = x_i x_j+x_k x_i+x_j x_k.
\end{align*}
\end{enumerate}

\begin{obs}\label{rmk:princ-realiz-sn} 
These Nichols algebras can be seen as Nichols 
algebras over $\s_{n}$ by a principal YD-realization (see 
\cite[Def. 3.2]{AG3}, \cite[Sec. 5]{MS}) of $(\mO_2^n,-1)$, 
$(\mO_2^n,\chi)$ over $\s_n$ or 
$(X,q)=(\mO_4^4,-1$) over $\s_4$; that is, one may describe
$\Bbbk X$ as a Yetter-Drinfeld module over $\s_{n}$.
In particular,
if we denote the elements of $\s_{n}$ by $h_{\tau}$ and
the elements of $X$ by $x_{\sigma}$ with $\sigma \in \mO^{n}_{k}$, $k=2,4$,
then 
the action and coaction of
$\s_{n}$ are determined by:
 \begin{equation}\label{eq:yd-realiz-sn}
 \delta(x_\tau)=h_{\tau} \ot x_\tau, \quad 
h_{\theta}\cdot x_\tau = \chi_\tau(h_{\theta})\, x_{\theta\trid \tau} \text{ for all }
\tau\in X,\ {\theta}\in \s_{n},
\end{equation}
where $(\chi_\tau)_{\tau \in X}$, with $\chi_\tau : \s_{n} \to \Bbbk^{\times}$, 
is a $1$-cocycle, i.~e., $\chi_\tau(\sigma\mu) = \chi_\tau(\mu)\chi_{\mu\trid\tau}(\sigma)$,
for all $\tau \in X, \sigma, \mu \in \s_{n}$, satisfying $\chi_x(y) = q_{yx}$ for all $x,y \in X$.
\end{obs}

\subsection{Classification of 
finite dimensional pointed Hopf algebras over $ \s_{3} $ and
$\s_{4}$}
\label{subsec:phasn}
In this subsection we present all finite dimensional pointed Hopf algebras
over $ \s_{3} $ and
$\s_{4}$  up to isomorphism. As before, first we introduce 
families of quadratic algebras. It turns out that these quadratic
algebras give all nontrivial liftings of bosonizations of
finite dimensional 
Nichols algebras.  
We follow \cite[Def. 3.7]{AG2} and \cite[3.9, 3.10]{GG}. 
Let $\Lambda, \Gamma, \lambda \in \Bbbk$ and
$t=(\Lambda, \Gamma) $.  
For $\theta, \tau \in \s_{n}$ denote $\theta\trid \tau = \theta \tau \theta^{-1} $ the 
conjugation in $\s_{n}$.

\begin{definition}\label{def:q-1}
$\mH(\mQ_n^{-1}[t])$ is the
algebra generated by $\{a_i, h_r :
i\in\mO_2^n,r\in\s_n\}$ satisfying the following relations for 
$r,s,j\in\s_n$ and $i\in \mO_2^n$:
\begin{align*}
h_e=1, \quad h_rh_s=h_{rs}, \qquad 
h_j a_i=-a_{j\trid i}h_j, 
\quad a_{(12)}^2=0,\\
a_{(12)} a_{(34)}+a_{(34)} a_{(12)}=
\Lambda(1-h_{(12)}h_{(34)}),\\
 a_{(12)} a_{(23)}+a_{(23)} a_{(13)}+a_{(13)}
a_{(12)}=\Gamma(1-h_{(12)}h_{(23)}).
\end{align*}
\end{definition}

\begin{definition}\label{def:qchi}
$\mH(\mQ_n^{\chi}[\lambda])$ is the algebra
generated by  
$\{a_i, h_r : i\in\mO_2^n,r\in\s_n\}$ satisfying the following relations for 
$r,s,j\in\s_n$ and $i\in \mO_2^n$:
\begin{align*}
h_e=1, \quad h_rh_s=h_{rs}, \quad
h_j a_i=\chi_i(j)\, a_{j\trid i}h_j, \quad
a_{(12)}^2=0,\\
 a_{(12)} a_{(34)}-a_{(34)} a_{(12)}=0\\
a_{(12)}a_{(23)}-a_{(23)}a_{(13)}-a_{(13)}a_{(12)}
=\lambda(1-h_{(12)}h_{(23)}).
\end{align*}
\end{definition}

\begin{definition}\label{def:qd}
$\mH(\mD[t])$ is the algebra generated by 
$\{a_i, h_r :
i\in\mO_4^4,r\in\s_4\}$ satisfying the following relations for 
$r,s,j\in\s_n$ and $i\in  \mO_4^4$:
\begin{align*}
h_e=1, \quad h_rh_s=h_{rs}, \quad 
h_j a_i=-a_{j\trid i}h_j, \quad 
a_{(1234)}^2=\Lambda(1-h_{(13)}h_{(24)}),\\
a_{(1234)}a_{(1432)}+a_{(1432)}a_{(1234)}=0,\\
a_{(1234)}a_{(1243)}+a_{(1243)}a_{(1423)}+a_{(1423)}a_{(1234)}
=\Gamma(1-h_{(12)}h_{(13)}).
\end{align*}
\end{definition}

\begin{obs}
For each quadratic lifting datum 
$\mQ = \mQ_n^{-1}[t], \mQ_n^{\chi}[\lambda], \mD[t]$, the  
algebra $\hq$ has a structure of a pointed Hopf algebra setting
\begin{equation}\label{eq:com-ql}
\com(h_t)=h_t\ot h_t\text{ and  }\com(a_i)= a_i\ot 1+h_i\ot a_i,\quad t\in
\s_{n}, i\in X.
 \end{equation}
Moreover, they
satisfy  that $\gr\hq={\toba}(X,q)\#\Bbbk \s_n$, 
with $n$ as appropriate see
\cite{GG}.
\end{obs}

The following theorem summarizes the classification of 
finite dimensional pointed
Hopf algebras over $\s_{3}$ and $\s_{4}$, see \cite{AHS}, 
\cite{GG}.

\begin{theorem}\label{thm:class-sn}
Let $H$ be a nontrivial 
finite dimensional pointed Hopf algebra with $G(H)=\s_{n}$.
\begin{enumerate}
 \item[$(i)$]  If 
$n=3$, then either 
$H\simeq \toba (\mO_2^3,-1)\# \Bbbk \s_{3}$ or
 $H\simeq \mH(\mQ_3^{-1}[(0,1)])$. 
 \item[$(ii)$] If 
$n=4$, then either $H\simeq  \toba(X,q)\# \Bbbk \s_{4}$ with
$(X,q) =  (\mO_2^4,-1)$, $(\mO_4^4,-1)$ or
$(\mO_2^4,\chi)$, or 
$H\simeq  \mH(\mQ_4^{-1}[t])$, or $H\simeq   \mH(\mQ_4^{\chi}[1])$, or 
$H\simeq  \mH(\mD[t])$ with 
$t\in\mathbb{P}_\Bbbk^1$.
\qed
\end{enumerate} 
\end{theorem}

\subsection{Cocycle deformations and 
pointed Hopf algebras over   
$ \s_{n} $}
In the following we construct multiplicative $2$-cocycles and show 
that some families of the pointed Hopf algebras 
$\mH(\mQ_n^{-1}[t])$ and $\mH(\mD[t])$ are cocycle deformations
 of bosonizations of Nichols 
algebras in $\ydsn$.
As a consequence, we provide the family of cocycles needed to 
construct all finite dimensional pointed Hopf algebras over 
$\s_{3}$ up to isomorphism.

\smallbreak
Let $X=\mO^{n}_{2}$ or $\mO^{4}_{4}$ and 
denote the generators of $\toba(X,-1) $
 by $ x_{\tau}$ with $ \tau \in X$.
For all
$\sigma, \tau \in X$, define the linear maps  $ d_{\tau}$ on $\Bbbk X$
by  $ d_{\tau}(x_{\mu}) = \delta_{\tau,\mu}$. By Subsection \ref{subsec:hoch-H-inv},
the following map is a Hochschild $2$-cocycle on $\toba(X,-1) $
$$ \eta = \sum_{\mu,\tau \in X} \alpha_{\tau,\mu}
d_{\tau}\ot d_{\mu}.$$ 
The proof of the following lemma follows by a direct computation.
\begin{lema}\label{lem:inv-bider-q-1} 
 $ \eta $ is $ \s_{n} $-invariant if and only if 
$\alpha_{\tau,\mu} = \alpha_{\theta\trid \tau,\theta \trid \mu}$
for all $\tau,\mu \in X$ and $\theta\in \s_{n}$.\qed
 \end{lema}

\begin{obs}\label{rmk:eta-inv-q-1}
Consider the set $\mathcal{T} =X\times X$.
Then $\s_{n}$, and in particular $X$, acts by conjugation on $\mathcal{T}$ by 
$\theta\cdot (\tau,\mu) = (\theta \trid \tau, \theta\trid \mu)$.  
If we set $\alpha: \mathcal{T} \to \Bbbk$ with
$\alpha (\tau, \mu) = \alpha_{\tau,\mu}$, then the coefficients
of $\eta$ are given by the function $\alpha$ and by Lemma \ref{lem:inv-bider-q-1},
$\eta$ is $\s_{n}$-invariant if and only if $\alpha$ is a class function,
\textit{i.e.} it is constant on each conjugacy class. 
Since $(\tau,\mu) $ is conjugate to $(\tau',\mu')$ if and only 
if $\tau\mu$ is conjugate to $\tau'\mu'$ in $\s_{n}$, if 
$\eta$ is $\s_{n}$-invariant, we may write in the case $X=\mO^{n}_{2}$
\begin{equation}\label{eq:rmk-eta-inv-q-1}
\eta = \beta_{\id} \sum_{\tau \in \mO^{n}_{2} }
d_{\tau} \ot d_{\tau}  + \beta_{(123)}  
\sum_{\substack{\tau,\mu \in \mO^{n}_{2}\\
\tau\mu \in \mO^{n}_{3}}}  
d_{\tau} \ot d_{\mu} +  \beta_{(12)(34)}
\sum_{\substack{\tau,\mu \in \mO^{n}_{2}\\
\tau\mu \in \mO^{n}_{2,2}}} 
d_{\tau} \ot d_{\mu}, 
\end{equation}
with $\beta_{\id}, \beta_{(123)}, \beta_{(12)(34)} \in \Bbbk$,
and in the case $X=\mO^{4}_{4}$
\begin{equation}
 \label{eq:rmk-eta-inv-q-1-4}
\eta = \gamma_{\id} \sum_{\tau \in \mO^{4}_{4} }
d_{\tau} \ot d_{\tau^{-1}}  + \gamma_{(123)}  
\sum_{\substack{\tau,\mu \in \mO^{4}_{4}\\
\tau\mu \in \mO^{4}_{3}}}  
d_{\tau} \ot d_{\mu} +  
\gamma_{(12)(34)}
\sum_{\tau \in \mO^{4}_{4}} 
d_{\tau} \ot d_{\tau},
 \end{equation}
with $\gamma_{\id}, \gamma_{(123)}, \gamma_{(12)(34)} \in \Bbbk$.
\end{obs}

Assume $\eta$ satisfies \eqref{eq:rmk-eta-inv-q-1} or 
\eqref{eq:rmk-eta-inv-q-1-4}. 
The next lemma states that 
the exponentiation
of the lifting of $\eta$ is a multiplicative 2-cocycle if
all coefficients $\beta$ or $\gamma$ are equal. Since the braiding
on $\toba(X,-1)$ is not symmetric, one needs to verify equations
 \eqref{eq1} and 
\eqref{eq2} on $V=\Bbbk X$.

\begin{lema}\label{lem:ql-1-cocycle}
Assume $ \eta = \sum_{\mu,\tau \in X} \alpha_{\tau,\mu}
d_{\tau}\ot d_{\mu}$ is $\s_{n}$-invariant. Then
it satisfies equations \eqref{eq1} and 
\eqref{eq2} 
if and only if $\alpha_{\tau,\mu} = \alpha_{\tau',\mu'}$
for all $\tau,\tau',\mu,\mu' \in X$.
In such a case, 
$ \sigma = e^{\tilde{\eta}} $ is a 
multiplicative $ 2 $-cocycle for $ \toba(X,-1)\# \Bbbk \s_{n} $.
\end{lema}

\pf By Lemma \ref{lem:cond-eps-biderivations}, we need 
only to verify equations \eqref{eq1} and 
\eqref{eq2} on $V=\Bbbk X$. 
Since
$\chi_{\tau}(h_{\mu}) = \sg(\mu) = -1$ for all 
$\tau, \mu \in \mO_2^{n}$ and 
$\chi_{\tau}(h_{\mu}) = 1$ for all 
$\tau, \mu \in \mO_4^{4}$ these
equations 
on 
$x_{r}, x_{s}$, $x_{t}$, $x_{u}$
with $r,s,t,u \in X $  equal:
\begin{align*}
\eqref{eq1}\quad & 
{\eta}(x_{r}, x_{s\trid t}){\eta}(x_{s},x_{u}) = 
 {\eta}(x_{r\trid s},x_{r\trid (t\trid u)})
{\eta}(x_{r}, x_{t}) ,\\
\eqref{eq2}\quad &  
{\eta}(x_{r}, x_{s\trid(t\trid u)}){\eta}(x_{s},x_{t}) 
=  
{\eta}(x_{r\trid s},x_{r\trid t})
\eta(x_{r}, x_{u}).
\end{align*}
It is clear that if $\eta = \lambda \sum_{\mu,\tau \in X}
d_{\tau}\ot d_{\mu} $ for some $\lambda \in \Bbbk$, then 
both equations are satisfied.
Conversely, assume $\eta$ satisfies \eqref{eq1} and \eqref{eq2}.
Since $t\trid - $ is a bijection for all $t\in X$, by \eqref{eq2}
we have that $\alpha_{r,s\trid(t\trid u)}{\alpha}_{s,t} = {\alpha}_{s,t}
{\alpha}_{r,u}$ for all $r,s,t,u \in X$.
If ${\alpha}_{s,t} \neq 0$ for some
$s,t\in X$, then ${\alpha}_{r,  u} = 
\alpha_{r, s \trid u}$ for all $r,s, u \in X$.
Since $\eta$ must satisfy \eqref{eq:rmk-eta-inv-q-1} or 
\eqref{eq:rmk-eta-inv-q-1-4},
it follows that  
$\eta = \lambda \sum_{\mu,\tau \in X}
d_{\tau}\ot d_{\mu} $ for some $\lambda \in \Bbbk$. The 
rest of the claim follows now by Lemma \ref{lem:cond-eps-biderivations}.
\epf

\begin{theorem}\label{thm:q-1-cocycle}
Let $ H = \toba(X,-1)\# \Bbbk \s_{n}$ 
and $\sigma = e^{\tilde{\eta}}$ 
be the multiplicative $2-cocycle$ given by \emph{Lemma \ref{lem:ql-1-cocycle}}
with $\eta = \frac{\lambda}{3} 
\sum_{\mu,\tau \in \mO_2^{n}}
d_{\tau}\ot d_{\mu} $ and $\lambda \in \Bbbk$. 
\begin{enumerate}
 \item[$(i)$] If $X= \mO^{n}_{2}$ then
$ H_{\sigma} \simeq  \mH(\mQ_3^{-1}[(0,\lambda)])$  for $n=3$ and
$ H_{\sigma} \simeq  
\mH(\mQ_n^{-1}[(2\lambda,3\lambda)])$  for $n\geq 4$.
 \item[$(ii)$] If $X= \mO^{4}_{4}$ then $ H_{\sigma} \simeq  
\mH(\mD[(\lambda,3\lambda)])$.
\end{enumerate}
In particular, $\mH(\mQ_3^{-1}[(0,\lambda)])$ is a cocycle deformation
of $H$ for all $\lambda \in \Bbbk$.
\end{theorem}

\pf 
As in the proof of Theorems \ref{thm:AI-cocycle} and 
\ref{thm:BIL-cocycle},
it suffices to show that the generators of $ H_{\sigma} $ satisfy the 
relations given in Definitions \ref{def:q-1} and \ref{def:qd}, respectively.
For $\tau \in X$, let $ a_{\tau} = 
x_{\tau}\# 1 \in H$.
Then by Lemma 
\ref{lem:eta=sigma} we have  
for all $\tau,\mu \in \mO^{n}_{2}$ that
\begin{align*}
a_{\tau}\cdot_{\sigma} a_{\mu}  
& = \eta(x_{\tau}, x_{\mu})
(1 - h_{\tau}h_{\mu}) +
a_{\tau}a_{\mu} 
= \lambda (1- h_{\tau\mu}) +
a_{\tau}a_{\mu}.
\end{align*}
Hence, if $X=\mO^{n}_{2}$ we get that
$a_{(12)}\cdot_{\sigma}a_{(12)}
= a_{(12)}^{2} +
\frac{\lambda }{3}(1- h_{(12)(12)}) = \frac{\lambda }{3} (1-h_{e}) = 0$
and 
\begin{align*}
& a_{(12)}\cdot_{\sigma}a_{(23)} +
a_{(23)}\cdot_{\sigma} a_{(13)} + 
a_{(13)}\cdot_{\sigma}a_{(12)}  = \\
& \qquad = a_{(12)}a_{(23)} +
a_{(23)}a_{(13)} + 
a_{(13)}a_{(12)} +
\frac{\lambda }{3}(1- h_{(12)(23)}) + 
\frac{\lambda }{3}(1- h_{(23)(13)}) + 
\frac{\lambda }{3}(1- h_{(13)(12)}) \\
& \qquad = \lambda (1-h_{(123)}) = \lambda(1-h_{(12)(23)}).
\end{align*}
Taking $\Gamma = \lambda$,
this implies that $ H_{\sigma} \simeq  
\mH(\mQ_3^{-1}[(0,\lambda)])$  if $n=3$,
since both algebras have the same dimension. For $n\geq 4$ we need
to verify the extra relation involving the 
product of two disjoint transpositions:
\begin{align*}
& a_{(12)}\cdot_{\sigma}a_{(34)} +
a_{(34)}\cdot_{\sigma} a_{(12)}=  a_{(12)}a_{(34)} +
a_{(34)}a_{(12)}+
\frac{\lambda }{3}(1- h_{(12)(34)}) + 
\frac{\lambda }{3}(1- h_{(34)(12)})= \\
& \qquad = \frac{2\lambda}{3} (1-h_{(12)(34)}).
\end{align*}
Thus taking $t= (\Lambda, \Gamma)  = 
(\frac{2\lambda}{3}, \lambda)$, 
we have that $ H_{\sigma} \simeq  
\mH(\mQ_n^{-1}[(2\lambda,3\lambda)])$.
Assume $X=\mO^{4}_{4}$, then 
\begin{align*}
& a_{(1234)}\cdot_{\sigma}a_{(1234)}
= a_{(1234)}^{2} +
\frac{\lambda }{3}(1- h_{(1234)(1234)}) = \frac{\lambda }{3}(1-h_{(13)(24)}),\\
& a_{(1234)}\cdot_{\sigma}a_{(1432)} +
a_{(1432)}\cdot_{\sigma} a_{(1234)} = a_{(1234)}a_{(1432)} +
a_{(1432)}a_{(1234)} +\\
& \quad \qquad +
\frac{\lambda }{3}(1- h_{(1234)(1432)}) + 
\frac{\lambda }{3}(1- h_{(1432)(1234)}) = \frac{2\lambda }{3}(1- h_{e})= 0,\\
& a_{(1234)}\cdot_{\sigma}a_{(1243)} +
a_{(1243)}\cdot_{\sigma} a_{(1423)} + 
a_{(1423)} \cdot_{\sigma} a_{(1234)} = \\
& \qquad =a_{(1234)}a_{(1243)} +
a_{(1243)}a_{(1423)} + 
a_{(1423)} a_{(1234)}+\\
& \quad \qquad +
\frac{\lambda }{3}(1- h_{(1234)(1243)}) + 
\frac{\lambda }{3}(1- h_{(1243)(1423)}) +
\frac{\lambda }{3}(1- h_{(1423)(1234)}) =
\lambda (1 - h_{(12)(13)}).
\end{align*}
Therefore, taking $t= (\Lambda, \Gamma)  = (\frac{\lambda}{3}, \lambda)$, 
we have that $ H_{\sigma} \simeq  \mH(\mD[(\lambda,3\lambda)]))$.
\epf

\begin{obs}
\emph{Cocycle deformations 
and the algebras $\mH(\mQ_n^{\chi}[\lambda])$.}
As shown in \cite{GIM}, the pointed Hopf algebras 
$\mH(\mQ_n^{\chi}[\lambda])$ are cocycle deformations of 
$ \toba(\mO_2^{n},\chi)\# \Bbbk \s_{n} $. Regrettably, our
construction using $ \s_{n} $-invariant linear functionals on $\Bbbk \mO_2^{n} \ot \Bbbk \mO_2^{n} $
only provides the trivial deformation.
\end{obs}

\subsection*{Acknowledgment}
Research of this paper was begun when the first author was visiting the 
Math. department at Saint Mary's University, Canada.
He thanks 
the second author and the people of the
department for their warm hospitality. The authors also wish to thank the referee
for the careful reading and the suggestions to improve the presentation of the paper.

\end{document}